\documentclass{amsart} 

\usepackage{amsmath,amssymb,amsthm} 
\usepackage{graphicx} 
\usepackage[arrow, matrix, curve]{xy} 
\usepackage{color} 
\usepackage{hyperref} 

\DeclareMathOperator{\tb}{tb}
\DeclareMathOperator{\rot}{rot}
\DeclareMathOperator{\lk}{lk}

\DeclareMathOperator{\de}{d}

\newcommand{\Z}{\mathbb{Z}}
\newcommand{\Q}{\mathbb{Q}}

\newcommand{\xist}{\xi_{\mathrm{st}}}

\newtheoremstyle{thm}{}{}{\itshape}{}{\bfseries}{\hfill\\}{ }{} 
\newtheoremstyle{definition}{}{}{}{}{\bfseries}{\hfill\\}{ }{} 

\theoremstyle{thm}
\newtheorem{Theorem}{Theorem}[section]
\newtheorem{thm}[Theorem]{Theorem}
\newtheorem{lem}[Theorem]{Lemma}

\newtheorem*{Theorem-ohne}{Theorem}
\newtheorem{con}[Theorem]{Conjecture}
\newtheorem{prob}[Theorem]{Problem}

\theoremstyle{definition}
\newtheorem{defi}[Theorem]{Definition}
\newtheorem{rem}[Theorem]{Remark}
\newtheorem{ex}[Theorem]{Example}

\begingroup\expandafter\expandafter\expandafter\endgroup
\expandafter\ifx\csname pdfsuppresswarningpagegroup\endcsname\relax
\else
\fi

\begin{document}


\title[The Legendrian knot complement problem]{The Legendrian knot complement problem} 

\author{Marc Kegel}

\address{Institut f\"ur Mathematik, Humboldt-Universit\"at zu Berlin, Unter den Linden 6, 10099 Berlin, Germany}
\email{kegemarc@math.hu-berlin.de}


\begin{abstract}
We prove that every Legendrian knot in the tight contact structure of the $3$-sphere is determined by the contactomorphism type of its exterior. Moreover, by giving counterexamples we show this to be not true for Legendrian links in the tight $3$-sphere. On the way a new user-friendly formula for computing the Thurston--Bennequin invariant of a Legendrian knot in a surgery diagram is given.
\end{abstract}

\date{\today} 



\thanks{The research of the author is supported by the Berlin Mathematical School}

\maketitle


\section{Introduction}

The \textit{knot complement problem}, first proposed in 1908 by Heinrich Tietze~\cite[Section~15]{T}, asks when a knot in a given $3$-manifold is determined by its complement. It had been open for eighty years until in 1989 Gordon and Luecke~\cite{GL} proved it to be true for every knot in $S^3$. In fact, they proved that non-trivial Dehn surgery along a non-trivial knot in $S^3$ cannot yield $S^3$ again. As a direct consequence, they obtain that every (tame) knot in $S^3$ is determined by its complement, while this is in general not true for links in $S^3$ and for knots in general manifolds. In Sections~\ref{section:complement} and~\ref{section:surgery} we recall the basic facts about the knot complement problem and Dehn surgery.

Here we want to consider the same problem for \textit{Legendrian knots} in \textit{contact $3$-manifolds}. The main result is a generalization of the result by Gordon and Luecke for Legendrian knots in $S^3$ with its standard tight contact structure $\xist$, roughly speaking it says the following (see Theorem~\ref{thm:contact2} for the precise statement).

\begin{thm} [Contact Dehn surgery theorem] \label{main:main0} 
If the result of a non-trivial rationally contact Dehn surgery along some Legendrian knot $K$ in $(S^3,\xist)$ yields again $(S^3,\xist)$ then $K$ has to be a Legendrian unknot whose Thurston--Bennequin invariant $\tb(K)$ and rotation number $\rot(K)$ are related by 
\begin{equation*}
\vert\tb(K)\vert=\vert\rot(K)\vert+1.
\end{equation*}
\end{thm}

Notice that a Legendrian unknot with classical invariants related by $\vert\tb(K)\vert=\vert\rot(K)\vert+1$ is obtained from the unique Legendrian unknot with $\tb=-1$ by a sequence of stabilizations all with the same sign. Very similar to the topological case it follows from Theorem~\ref{main:main0} that a Legendrian knot in $(S^3,\xist)$ is determined by the contactomorphism type of its exterior (i.e.\ the complement of an open \textit{standard neighborhood }of the Legendrian knot). For the precise statement see Theorem~\ref{thm:contact1}.

\begin{thm} [Legendrian knot exterior theorem] \label{main:main1} 
Two Legendrian knots in $(S^3,\xist)$ are isotopic (as Legendrian knots) if and only if their exteriors are contactomorphic.
\end{thm}

This result implies that all invariants of Legendrian knots are actually invariants of the contactomorphism type of its exterior. It remains unclear if Theorem~\ref{main:main1} holds also for the complements instead of the exteriors (see discussion after  Theorem~\ref{thm:contact1}). The Legendrian knot complement problem was also mentioned in~\cite{Et2}. 

With similar ideas as in the present article, one can obtain similar results as Theorem~\ref{main:main0} and~\ref{main:main1} also for transverse knots in $(S^3,\xist)$~\cite{Ke17b}.

In Sections~\ref{section:contactcomplement} and~\ref{section:contactsurgery} we recall the definition of \textit{contact Dehn surgery} and explain how to deduce Theorem~\ref{main:main1} from Theorem~\ref{main:main0}. 

One way of proving Theorem~\ref{main:main0} is to consider a Legendrian knot $L$ in the exterior of the surgery unknot $K$ stabilized in both ways. Then this Legendrian knot $L$ represents also a Legendrian knot in the surgered manifold. By showing that the Thurston--Bennequin invariant of $L$ in the surgered contact manifold violates Bennequin-type inequalities (which hold only true in tight contact manifolds) we can deduce that the surgered contact manifold has to be overtwisted and consequently cannot be $\xist$.

Formulas for computing the classical invariants $\tb$ and $\rot$ in contact surgery diagrams are given in~\cite{LOSS,GO,C}. However, these formulas only work for contact $(\pm1)$-surgeries and since we are concerned with a general contact $r$-surgery these formulas cannot be used here. Therefore, we present in Section~\ref{section:LOSS} a new formula to compute the Thurston--Bennequin invariant of a Legendrian knot presented in a general rationally contact $r$-surgery diagram along Legendrian knots. Moreover, our formula works also in general contact manifolds, for contact surgeries along non-Legendrian knots and also simplifies the computations in contact $(\pm1)$-surgery diagrams.

Building up on this and using the formulas in~\cite{LOSS,GO,C} one can also get formulas for computing rotation numbers of Legendrian knots, self-linking numbers of transverse knots and the $\de_3$-invariant of the resulting contact manifold in contact $(1/n)$-surgery diagrams along Legendrian knots~\cite{DK}. In~\cite{DKK,DK2} similar formulas are given for computing the classical invariants of Legendrian knots sitting on the page of a contact open book. 

Moreover, in Section~\ref{counterexample} we explain how to do a crossing change in a Legendrian knot diagram with help of a contact Dehn surgery (a so-called \textit{contact Rolfsen twist}). With that, it is easy to construct counterexamples to the Legendrian link exterior problem in the tight contact structure of $S^3$, i.e. two non-equivalent Legendrian links with contactomorphic exteriors. 

Finally, in Section~\ref{topgenmfd} and~\ref{section:other} we give a short discussion about the Legendrian knot exterior problem in other manifolds. It turns out that in the topological setting the knot complement problem in a general manifold is equivalent to the non-existence of an \textit{exotic cosmetic Dehn surgery} resulting in this manifold. In the contact setting, it is not clear if this equivalence is true. In Section~\ref{section:other} we give a short discussion on this topic and present examples of \textit{exotic cosmetic contact Dehn surgeries}.



\subsubsection*{Acknowledgment.} This work is part of my Ph.D.\ thesis~\cite{Ke17a} which was partially supported by the DFG Graduiertenkolleg 1269 "Global Structures in Geometry and Analysis". I would like to thank Sebastian Durst, Hansj\"org Geiges, Mirko Klukas, Sinem Onaran, and the referee for reading earlier versions of this article and for many useful discussions and suggestions.


\section{The Knot Complement Problem}
	\label{section:complement}
	
All links are assumed to be tame and considered up to (coarse) equivalence.
\begin{defi}[Coarse equivalence] \label{def:coarseequiv}
Let $L_1$ and $L_2$ be two links in an oriented closed $3$-manifold $M$. Then $L_1$ is \textbf{(coarse) equivalent} to $L_2$, if there exists a homeomorphism $f$ of $M$
\begin{align*}
	f\colon M\longrightarrow M\\
	L_1 \longmapsto L_2
\end{align*}
that maps $L_1$ to $L_2$. Then we write $L_1 \sim L_2$.
\end{defi}

\begin{rem}[Coarse equivalence vs. oriented coarse equivalence vs. isotopy]\label{rem:coarseequiv}
The (coarse) equivalence is a weaker condition than the equivalence of knots up to isotopy. For example, there is a reflection of $S^3$ that maps the left-handed trefoil to the right-handed trefoil, so these two knots are (coarse) equivalent, but one can show that the left-handed trefoil is not isotopic to the right-handed trefoil (i.e. there is no such homeomorphism isotopic to the identity). 

One also can consider the \textbf{oriented (coarse) equivalence}, that means equivalence where only orientation-preserving homeomorphisms of $M$ are allowed. In $S^3$ oriented equivalence is equivalent to isotopy (because in $S^3$ every orientation preserving homeomorphism is isotopic to the identity) but in general manifolds this is a weaker condition than isotopy.
\end{rem}

A first observation is, that if two links $L_1$ and $L_2$ are equivalent, then their complements are homeomorphic. The following question is called the \textbf{link complement problem} (or for one component links, the \textbf{knot complement problem}): 

\begin{prob}[Link complement problem]
Are two links in the same manifold with homeomorphic complements equivalent?
\end{prob}

Link complements are non-compact, but often it is much easier to work with compact manifolds. Therefore, pick some regular (closed) neighborhood $\nu L$ of a link $L$ in $M$ and call the complement $M\setminus\mathring{\nu L}$ of the interior $\mathring{\nu L}$ of this neighborhood the \textbf{exterior} of $L$. 

The corresponding problem whether the equivalence class of a link is determined by the homeomorphism type of its exterior is called the \textbf{link exterior problem}. (Actually, this was the problem asked by Tietze in~\cite[Section 15]{T}.) By work of Edwards~\cite[Theorem 3]{E} these two problems are equivalent (compare also the discussion after Theorem~\ref{thm:contact1}).

\begin{ex}[The Whitehead links]\label{ex:Whitehead}
The first counterexample was given in 1937 by Whitehead~\cite{W}. He considered the following two links $L_1$ and $L_2$ in $S^3$, now called \textbf{Whitehead links}, see Figure~\ref{fig:whitehead}.

If one deletes one component out of $L_1$ then the remaining knot is an unknot. But if one deletes the unknot $U$ out of $L_2$ then the remaining knot is a trefoil. So $L_1$ cannot be equivalent to $L_2$. But the exteriors of these links are homeomorphic, as one can see as follows: Consider the exterior $S^3\setminus\mathring{ \nu U}$ of the unknot $U$. Then cut open this $3$-manifold along the Seifert disk of the unknot $U$, make a full $2\pi$-twist and re-glue the two disks together again. This is called a \textbf{Rolfsen twist} and describes a homeomorphism of the link exteriors.

By twisting several times along $U$ one can even get infinitely many non-equivalent links all with homeomorphic exteriors.
\end{ex}

\begin{figure}[htbp] 
\centering
\def\svgwidth{0,9\columnwidth}
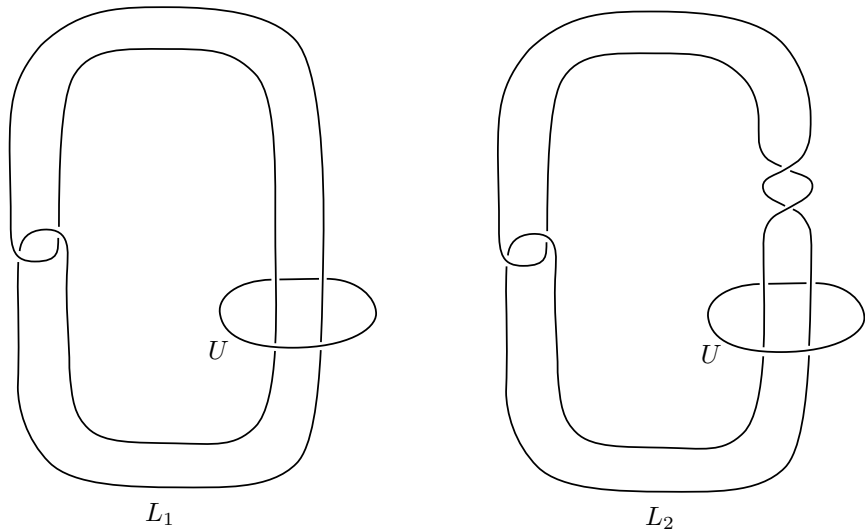
\caption{Two non-equivalent links with homeomorphic complements}
\label{fig:whitehead}
\end{figure}	
The next natural question would be to ask if this holds on the level of knots. This is the so-called knot complement theorem by Gordon-Luecke~\cite{GL}.

\begin{thm} [Knot complement theorem by Gordon-Luecke] \label{thm:gordon-luecke1} 
Let $K_1$ and $K_2$ be two knots in $S^3$ with homeomorphic complements, then $K_1$ is equivalent to $K_2$.
\end{thm}

A starting point for proving this theorem was to translate it into a problem concerning Dehn surgery.


	\section{Dehn Surgery}
	\label{section:surgery}
In this section, we recall the definition of Dehn surgery, which is a very effective construction method for $3$-manifolds (for more information see~\cite[Chapter~VI]{PS} or~\cite[Chapter~9]{R}). Roughly speaking one cuts out the neighborhood of a knot and glues a solid torus back in a different way to obtain a new $3$-manifold. More precisely:

\begin{defi}[Dehn surgery]  
Let $K$ be a knot in a closed oriented $3$-manifold $M$. Take a non-trivial simple closed curve $r$ on $\partial (\nu K)$ and a homeomorphism $\varphi$, such that
\[\begin{array}{ccc}
\varphi\colon \partial(S^1\times D^2)&\longrightarrow& \partial (\nu K)\\
\mu_0	:= \{\text{pt}\}\times\partial D^2  &\longmapsto& r.
\end{array}\]
Then define
\[\begin{array}{rccccl}
M_k(r)&:=& S^1\times D^2 &+& M\setminus\mathring{\nu K}&\big/_\sim ,\\
&&\partial(S^1 \times D^2)\ni p&\sim& \varphi(p)\in \partial(\nu K).&
\end{array}\]
One says that $M_K(r)$ is obtained out of $M$ by \textbf{Dehn surgery} along $K$ with \textbf{slope}~$r$.
\end{defi}

One can easily show that $M_K(r)$ is again a $3$-manifold independent of the choice of $\varphi$ (see~\cite[Chapter~9.F]{R}). So to specify $M_K(r)$ one only has to describe the knot $K$ (in the $3$-manifold $M$) together with the slope~$r$. To do this effectively one observes that there are two special kinds of curves on $\partial (\nu K)$:
\begin{itemize}
	\item The \textbf{meridian} $\mu$: A simple closed curve on $\partial(\nu K)$, that is non-trivial on $\partial(\nu K)$, but trivial in $\nu K$.
	\item The \textbf{longitudes} $\lambda$: Simple closed curves on $\partial(\nu K)$, that are non-trivial on $\partial(\nu K)$ and intersect $\mu$ transversely exactly once.
\end{itemize}
The curves shall be oriented in such a way, that the pair $(\mu,\lambda)$ represents the positive orientation of $\partial(\nu K)$ in $M$. One can show that the meridian $\mu$ is up to isotopy uniquely determined. But for the longitudes, there are different choices. For a given longitude $\lambda$ there are infinitely many other longitudes given by $\tilde{\lambda}=\lambda + q\mu $, for $q\in\Z$. Given such a longitude $\lambda$ one can write $r$ uniquely as
\begin{equation*}
r=p\mu+q\lambda, \,\text{for} \, p,  q \,\text{coprime},
\end{equation*}
where we regard $r$, $\mu$ and $\lambda$ as homology classes in $H_1(\partial(\nu K);\Z)$. 

For nullhomologous knots $K$ (i.e.\ knots that bounds a compact so-called \textbf{Seifert surface}) there is a preferred longitude, the so-called \textbf{surface longitude} $\lambda_s$, obtained from $K$ by pushing it into the direction of some Seifert surface. If we express the slope $r$ with respect to the surface longitude $\lambda_s$, i.e. $r=p\mu+q\lambda_s$, then the rational number $p/{q}\in\Q\cup\{\infty\}$ is called the \textbf{(topological) surgery coefficient}. And in fact, it is easy to show that for a slope $r=p\mu+q\lambda$ (and a given longitude $\lambda$) the surgered manifold $M_K(r)$ is already determined by the rational number $p/q\in\Q\cup\{\infty\}$ (see for example~\cite[Section~9.G]{R}).  From now on, depending on the context, we will denote by $r$ the slope or the corresponding surgery coefficient.

\begin{ex}[Surgeries along the unknot]\label{ex:surgeryUnknot}
We want to describe surgeries along the unknot $U$ in $S^3$. Observe that the exterior $S^3\setminus \mathring{\nu U}$ is again homeomorphic to a solid torus $S^1\times D^2$, this corresponds to the trivial genus-$1$ Heegaard splitting of $S^3$ (see for example~\cite[Example~8.5.]{PS}). Write
\[\begin{array}{ccc}
T_1:=&\nu U&\cong S^1\times D^2,\\
T_2:=&S^3\setminus \mathring{\nu U}&\cong S^1\times D^2.
\end{array}\]
Then the surface longitude $\lambda_1$ of $T_1$ is the meridian $\mu_2$ of $T_2$ and we choose the longitude $\lambda_2$ of $T_2$ to be the meridian $\mu_1$ of $T_1$ (see for example~\cite[Figure 8.7.]{PS}). So for surgeries along the unknot in $S^3$ one can write the slope $r$ uniquely as $r=p\mu_1+q\lambda_1$.

Now we want to show that $S^3_U(\mu_1+q\lambda_1)$ is homeomorphic to $S^3$. Therefore, one first considers the so-called \textbf{trivial Dehn surgery} $S^3_U(\mu_1)$, where one cuts out a neighborhood of the knot and glues it back in the same way as before. So the manifold is not changing, and in this case, it is again $S^3$. Then the idea is to do again a Rolfsen twist along the unknot to obtain a homeomorphism from $S^3\cong S^3_U(\mu_1)$ to $S^3_K(\mu_1+q\lambda_1)$. Therefore, consider the diagram that will be specified in the following.\\
\hfill\\
\begin{xy}
(5,50)*+{S^3}="a";(11,50)*+{\cong}="a1"; (20,50)*+{S^3_U(\mu_1)}="b"; (32,50)*+{:=}="b1"; (45,50)*+{S^1\times D^2}="c"; (75,50)*+{+}="d"; (102,50)*+{T_2}="e";(115,50)*+{\big/_\sim}="e1";%
(47,45)*+{\mu_0}="f"; (100,45)*+{\mu_1}="g";%
(47,40)*+{\lambda_0}="h"; (100,40)*+{\lambda_1}="i";%
(47,10)*+{\lambda_0}="j"; (100,10)*+{\lambda_1}="k";%
(75,25)*+{\circlearrowright}="k1";%
(47,5)*+{\mu_0}="l"; (95,5)*+{\mu_1+q\lambda_1}="m";%
(20,0)*+{S^3_U(\mu_1+q\lambda_1)}="n";(34,0)*+{:=}="n1"; (45,0)*+{S^1\times D^2}="o"; (75,0)*+{+}="p"; (102,0)*+{T_2}="q";(115,0)*+{\big/_\sim}="q1";%
{\ar@{-->}_{\cong} "b";"n"};{\ar@{->}@/_1pc/ _{\operatorname{Id}} "c";"o"};{\ar@{->}@/^1pc/^h "e";"q"};%
{\ar@{|->}_{\varphi_1} "f";"g"};{\ar@{|->} "h";"i"};{\ar@{|->}_{\varphi_2} "j";"k"};{\ar@{|->} "l";"m"};
\end{xy}\\
\hfill\\
First one chooses the gluing maps $\varphi_i$ such that they map the meridians $\mu_0$ as determined by the slope. By this, the manifolds are fixed, but the maps $\varphi_i$ are not. There are many possibilities to what a longitude $\lambda_0$ can map, but the homeomorphism type of the resulting manifold is not affected by this. In this example, one can choose the maps $\varphi_i$ such that they map $\lambda_0$ to $\lambda_1$. 

To construct a homeomorphism between the two resulting manifolds, one uses on the $S^1\times D^2$-factor the identity map. If one finds a homeomorphism $h$ of the $T_2$-factor such that the diagram commutes, then these two maps fit together to a homeomorphism of the whole manifolds. For the map $h$ one can choose a $q$-fold Dehn-twist of the solid torus $T_2$, i.e.
\[\begin{array}{rrcl}
h\colon& T_2&\longrightarrow& T_2\\
	&\mu_1=\lambda_2 &\longmapsto& \mu_1+q\lambda_1=\lambda_2+q\mu_2\\
	&\lambda_1=\mu_2 & \longmapsto &\lambda_1=\mu_2.
\end{array}\]
So the diagram gives rise to a homeomorphism between the two manifolds.
\end{ex}

\begin{rem}[The homology of the surgered manifolds]\label{rem:hom}
For all other surgeries along the unknot, one computes the homology as 
\begin{equation*}
H_1\big(S^3_U(p\mu_1+q\lambda_1);\mathbb{Z}\big)=\Z_p.
\end{equation*}
So the surgeries from the above example are the only surgeries along the unknot that lead again to $S^3$. In fact, these are the only non-trivial surgeries along an arbitrary knot in $S^3$ that yield again $S^3$, as the following deep theorem shows.
\end{rem}

\begin{thm} [Surgery theorem by Gordon-Luecke~\cite{GL}]\label{thm:gordon-luecke2}
Let $K$ be a knot in $S^3$. If $S^3_K(r)$ is homeomorphic to $S^3$ for some $r\neq \mu$, then $K$ is equivalent to the unknot $U$.
\end{thm}

Theorem~\ref{thm:gordon-luecke1} now follows easily from Theorem~\ref{thm:gordon-luecke2}. The connection is as follows. Assume first that the meridian $\mu$ of the knot $K$ is marked on the boundary of the knot exterior. Then the knot can be recovered easily because there is a unique way to glue in a solid torus by requiring that the meridian of the solid torus should map to the meridian of the knot $K$. (First use the Alexander trick in dimension $2$ to fill in a unique disk bounding the meridian. Then the boundary of the resulting object is a $2$-sphere. By again using the Alexander trick, this time in dimension $3$, there is a unique way to fill this $2$-sphere with a $3$-ball.) Then, one gets back the knot $K$ as the spine of the newly glued-in solid torus. In the language of Dehn surgery, this was nothing but a trivial Dehn surgery along the knot $K$ to get back $S^3$.

If the meridian is not given, then the question is of course how many different curves on the boundary of the knot exterior give back $S^3$ by doing Dehn surgery along $K$ with this slope. The surgery theorem says exactly that this is only possible for the unknot. To be more precise:

\begin{proof}[Proof of Theorem~\ref{thm:gordon-luecke1}]\hfill\\
Choose a homeomorphism 
\begin{equation*}
h\colon S^3\setminus\mathring{\nu K_1}\longrightarrow S^3\setminus\mathring{\nu K_2}
\end{equation*}
and write
\begin{equation*}
S^3 \cong S^3_{K_1}(\mu_1)=\, S^1\times D^2\, +\,   S^3\setminus\mathring{\nu K_1}\,\big/_\sim,
\end{equation*}
where the gluing map is $\varphi_1\colon\mu_0\mapsto\mu_1$. Then consider the surgery along $K_2$ with respect to the composition of maps
\[\begin{array}{ccccc}
\partial\big(S^1\times D^2\big)& \stackrel{\varphi_1}{\longrightarrow}&\partial \big(S^3\setminus\mathring{\nu K_1}\big)&\stackrel{h}{\longrightarrow} &\partial\big(S^3\setminus\mathring{\nu K_2}\big)\\
\mu_0&\longmapsto&\mu_1&\longmapsto& h(\mu_1):=r_2.
\end{array}\]
To determine the homeomorphism type of this new manifold $S^3_{K_2}(r_2)$ look at the following diagram:\\
\hfill\\
\begin{xy}
(10,40)*+{S^3}="a";(15,40)*+{\cong}="a1"; (25,40)*+{S^3_{K_1}  (\mu_1)}="b"; (35,40)*+{:=}="b1"; (45,40)*+{S^1\times D^2}="c"; (75,40)*+{+}="d"; (102,40)*+{S^3\setminus\mathring{\nu K_1}}="e";(115,40)*+{\big/_\sim}="e1";%
(47,35)*+{\mu_0}="f"; (100,35)*+{\mu_1}="g";%
(75,20)*+{\circlearrowright}="k1";%
(47,5)*+{\mu_0}="l"; (94,5)*+{r_2=h(\mu_1)}="m";%
(25,0)*+{S^3_{K_2}(r_2)}="n";(35,0)*+{:=}="n1"; (45,0)*+{S^1\times D^2}="o"; (75,0)*+{+}="p"; (102,0)*+{S^3\setminus\mathring{\nu K_2}}="q";(115,0)*+{\big/_\sim}="q1";%
{\ar@{-->}_{f} "b";"n"};{\ar@{->}@/_1pc/ _{\operatorname{Id}} "c";"o"};{\ar@{->}@/^1pc/^h "e";"q"};%
{\ar@{|->}_{\varphi_1} "f";"g"};{\ar@{|->}^{h\circ\varphi_1} "l";"m"};
\end{xy}\\
\hfill\\
With similar arguments as in Example~\ref{ex:surgeryUnknot} this induces a homeomorphism $f$ from $S^3\cong S^3_{K_1}  (\mu_1)$ to $S^3_{K_2}(r_2)$ and with Theorem~\ref{thm:gordon-luecke2} it follows that $r_2$ is equal to $\mu_2$ or $K_2$ is equivalent to the unknot $U$. 

If $r_2=\mu_2$, then the surgery $S^3_{K_2}(r_2)$ is the trivial surgery, so the spines
\begin{equation*}
 S^1\times \{0\}\subset S^1  \times D^2\subset S^3
\end{equation*}
of the new solid tori are equal to the knots $K_i$. Therefore, $f$ sends $K_1$ to $K_2$. 

In the other case ($K_2\sim U$) one does the same thing again but with $K_1$ and $K_2$ reversed, then it follows that $K_1\sim U\sim K_2$.
\end{proof}

\begin{rem}[Oriented knot complement theorem]
Exactly the same works also with orientations. Theorem~\ref{thm:gordon-luecke2} also holds for oriented homeomorphism from $S^3_K(r)$ to $S^3$ and oriented equivalence from $K$ to $U$. Then exactly the same proof as before shows that two knots with orientation preserving homeomorphic complements are orientation preserving equivalent. For $S^3$ this is the same as isotopic knots (see also Remark~\ref{rem:coarseequiv}).
\end{rem}


	\section{The Legendrian knot complement problem}
\label{section:contactcomplement}

Now we want to generalize this proof to the case of \textbf{Legendrian knots} in the unique tight contact structure $\xi_{st}$ on $S^3$, i.e. the tangent line to the knots lies always in the $2$-plane field given by the contact structure (see~\cite{G} for all basics about contact geometry and Legendrian knots). We want to consider Legendrian links up to (coarse) equivalence like in Definition~\ref{def:coarseequiv} and we only consider cooriented contact structures.

\begin{defi}[Coarse equivalence]  \label{defi:equi} 
Let $L_1$ and $L_2$ be two Legendrian links in a closed contact $3$-manifold $(M,\xi)$. Then $L_1$ is \textbf{(coarse) equivalent} to $L_2$ if there exists a contactomorphism $f$ of $(M,\xi)$
\[\begin{array}{rccc}
	f\colon &(M,\xi)&\longrightarrow &(M,\xi)\\
	&L_1 &\longmapsto& L_2,
\end{array}\]
that maps $L_1$ to $L_2$. Then we write $L_1 \sim L_2$.
\end{defi}

\begin{rem}[Coarse equivalence vs Legendrian isotopy]
The (coarse) equivalence is, in general, a weaker condition than the equivalence given by Legendrian isotopy (for example in overtwisted contact structures on $S^3$~\cite{Vo16}). But it is known that in $(S^3,\xi_{st})$ this two concepts are the same, since every contactomorphism of $(S^3,\xi_{st})$ is isotopic to the identity~\cite{El92} (compare also the discussion in~\cite[Section~4.3]{EF}).
\end{rem}

It is a standard fact that every Legendrian knot $K$ in a general contact $3$-manifold $(M,\xi)$ has a so-called \textbf{standard neighborhood} $\nu K$ in $(M,\xi)$ which is contactomorphic to
\begin{equation*}
\big(S^1\times D^2,\ker(\cos n\theta \,dx-\sin n\theta\,dy)  \big),
\end{equation*}
with $n$ a non-vanishing integer (which corresponds to the chosen topological identification of the neighborhood $\nu K$ with a copy of $S^1\times D^2$), $S^1$-coordinate $\theta$ and Cartesian coordinates $(x,y)$ on $D^2$. This contactomorphism maps $K$ to $S^1\times\{0\}$ (see~\cite[Example~2.5.10]{G}). It is easy to show that the boundary of this standard neighborhood is a convex surface. When we write $\nu K$ for a Legendrian knot $K$, we always mean that $\nu K$ is such a standard neighborhood and analogously to Section~\ref{section:complement} we call the complement $(M\setminus\mathring{\nu K},\xi)$ of the interior of such a standard neighborhood the \textbf{exterior} of the Legendrian knot $K$. 

Again, it follows by restricting the contactomorphism from Definition~\ref{defi:equi} to the knot exteriors that two equivalent Legendrian links have contactomorphic exteriors. The question if the reverse implication is also true we call the \textbf{Legendrian link exterior problem}. The following precise statement of Theorem~\ref{main:main1} says that this is true for Legendrian knots in $(S^3,\xi_{st})$.

\begin{thm} [Legendrian knot exterior theorem] \label{thm:contact1} 
Let $K_1$ and $K_2$ be two Legendrian knots in $(S^3,\xist)$ with contactomorphic exteriors. Then $K_1$ is equivalent to $K_2$.
\end{thm}

\begin{rem}[Unoriented Legendrian links]
Here $K_1$ and $K_2$ are understood to be unoriented Legendrian knots because the exterior of a knot cannot see its orientation. But if one fixes an oriented longitude of the knot in its exterior, the same result holds also for oriented Legendrian knots.
\end{rem}

For Legendrian links in $(S^3,\xist)$ Theorem~\ref{thm:contact1} is in general wrong. In Section~\ref{counterexample} we will give some examples of Legendrian links in $(S^3,\xist)$ that are \textbf{not} determined by the contactomorphism type of their exteriors and in Section~\ref{section:other} we will give a discussion about the Legendrian knot exterior problem in general contact manifolds.

In contrast to the topological setting, it is not clear if Theorem~\ref{thm:contact1} is also true for the knot complements instead of the knot exteriors.

For a general link $K$ in a general $3$-manifold $M$, we can compare the relation of the equivalence class of this link to the homeomorphism type of its complement, its open exterior, and its closed exterior. These relations are shown in Figure~\ref{fig:diagram}.

\begin{figure}[htbp] 
\centering
\def\svgwidth{0,96\columnwidth}
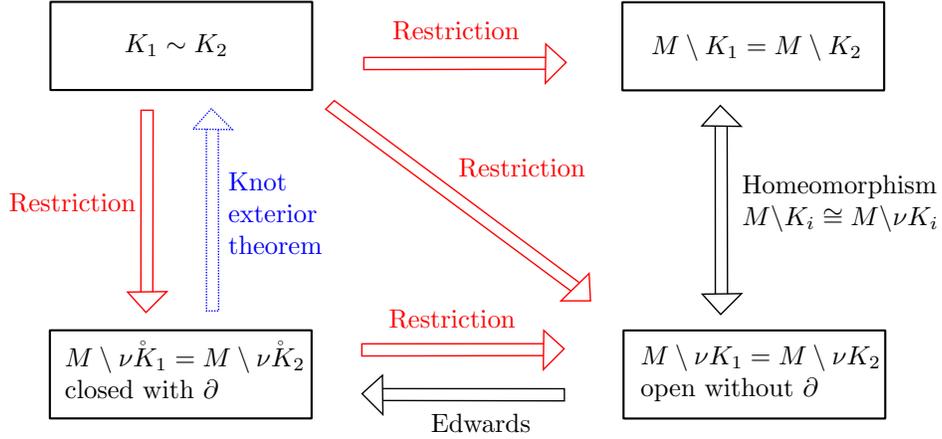
\caption{Relation between different types of classifications}
\label{fig:diagram}
\end{figure}

If two links $K_1$ and $K_2$ are equivalent then one can restrict the homeomorphism of $M$ mapping one link to the other to subsets of $M$. Since such a homeomorphism has to map a tubular neighborhood of one link to a tubular neighborhood of the other link one gets the red implications in Figure~\ref{fig:diagram} by restriction.

Moreover, it is easy to construct a homeomorphism of the complement of a link to the open exterior of the same link. (Take a homeomorphism of a punctured disk to a half-open annulus in every $D^2$-slice of the tubular neighborhood.) 

The mentioned non-trivial theorem of Edwards~\cite{E} states that two $3$-manifolds with boundary are homeomorphic if and only if their interiors are homeomorphic. Therefore, it is also equivalent to consider closed or open link exteriors.

And finally, the blue implication is exactly the statement of the link exterior problem, which is in general not true, but holds for knots in $S^3$.

Now let $K_1$ and $K_2$ be Legendrian links in some contact $3$-manifold $(M,\xi)$. First recall that all standard neighborhoods of Legendrian links are contactomorphic. Therefore, for Legendrian links, the exteriors are independent of the tubular neighborhoods as in the topological case. (This is not true for transverse knot, see~\cite{Ke17b}).

If two Legendrian links are equivalent then one can again restrict the contactomorphism to subsets and gets as in the topological case the red implications in Figure~\ref{fig:diagram} by restriction. 

But whether the black implications from Figure~\ref{fig:diagram} also hold in contact geometry is not clear. Contact structures on manifolds which are interiors of compact $3$-manifolds with boundary are studied by Eliashberg~\cite{E91}, Makar-Limanov~\cite{ML} and Tripp~\cite{Tr06}. With the methods developed there, it should be possible to study if the theorem of Edwards also holds for Legendrian and transverse links.

But the relation between the complements of Legendrian or transverse links and their open exteriors remains mysterious. The topological diffeomorphism between these sets do not preserve the contact structures, it remains open if such a contactomorphism exists or not.

Finally, the blue implication is the Legendrian or transverse link exterior problem, which is again in general not true, but holds for Legendrian and transverse knots in $(S^3,\xi_{st})$.

	
	\section{Contact Dehn Surgery}	
	\label{section:contactsurgery}
		
To generalize the proof from the topological setting to Legendrian knots in contact manifolds, we first want to recall the well-known definition of contact Dehn surgery along Legendrian knots.
	
\begin{defi}[Contact longitude] 
Let $K$ be a Legendrian knot in a contact $3$-manifold $(M,\xi)$. Then there is a distinguished, so-called \textbf{contact longitude} $\lambda_c$ on $\partial (\nu K)$, given by pushing $K$ in a direction transverse to the contact planes, for example in the direction of the Reeb vector field. 
\end{defi}	

If a Legendrian knot $K$ in a contact $3$-manifold $(M,\xi)$ is also nullhomologous it admits two distinguished longitudes, the surface longitude $\lambda_s$ and the contact longitude $\lambda_c$. These two longitudes differ only by an integer number of meridians $\mu$. This integer number is called the \textbf{Thurston--Bennequin invariant} $\tb(K)$, i.e.
\begin{equation*}
\lambda_c=\tb(K)\mu+\lambda_s \in H_1(\partial \nu K;\Z).
\end{equation*}
Now, we want to do Dehn surgery along a Legendrian knot $K$ with respect to the contact longitude $\lambda_c$. Again, if we write a slope $r$ as $r=p\mu+q\lambda_c$ the topological type of the surgered manifold $M_K(r)$ is already determined by the rational number $r_c=p/q\in\Q\cup\{\infty\}$, called the \textbf{contact surgery coefficient}. Sometimes, if $K$ is also nullhomologous, we want to compute the topological surgery coefficient~$r_s$ (with respect to the surface longitude $\lambda_s$) from the contact surgery coefficient $r_c$. This can be done via the formula $r_{s}=r_c+\operatorname{tb}(K)$.

Then we can extend the old contact structure on the knot exterior to a global contact manifold of the surgered manifold.

\begin{thm}[Contact Dehn surgery]\label{thm:contsurg}
Let $K$ be a Legendrian knot in a contact $3$-manifold $(M,\xi)$. \\
(1) Then $M_K(r)$ carries a (non-unique) contact structure $\xi_K(r)$, which coincides with the old contact structure $\xi$ on $M\setminus\mathring{\nu K}$.\\
(2) For $r\neq\pm \lambda_c$ one can choose $\xi_K(r)$ to be tight on the new glued-in solid torus.\\
(3) For $r=\mu+q\lambda_c$ this tight contact structure on this new solid torus is unique.
\end{thm}

For any of these choices for the contact structure $\xi_K(r)$ we say that the contact manifold $(M_K(r),\xi_K(r))$ is obtained from $(M,\xi)$ by \textbf{contact Dehn surgery} along the Legendrian knot $K$ with slope $r$. We give a proof of Theorem~\ref{thm:contsurg} following~\cite{DG} (see~\cite{Ke} for more details).

\begin{proof} [Proof of Theorem~\ref{thm:contsurg}] \hfill\\
Since every Legendrian knot $K$ looks locally the same we can think of $K$ as $S^1\times\{0\}$ in the standard model $S^1\times D^2_R$ (where $D^2_R$ denotes a disk with sufficiently big radius~$R$) with contact structure given as the kernel of 
\begin{equation*}
\cos(n\theta)dx-\sin(n\theta)dy. 
\end{equation*}
We choose $\nu K$ as $S^1\times D^2\subset S^1\times D^2_R$. Observe that $\partial (\nu K)$ is a convex surface with two dividing curves parallel to $\lambda_c$. 

Next, we delete $\nu K$ and glue back a new (topological) copy of $S^1\times D^2$ via a gluing diffeomorphism $\varphi$ of the boundaries with
\begin{align*}
\varphi\colon \partial(S^1\times D^2)&\longrightarrow \partial(\nu K)\\
\mu_0&\longmapsto r=p\mu+q\lambda_c.
\end{align*}

If we can find a contact structure $\xi'$ on the newly glued-in $S^1\times D^2$ with convex boundary and two dividing curves mapping under $\varphi$ to the dividing curves of $\partial(\nu K)$ then $\xi'$ glues together with the old contact structure to a global contact structure on the surgered manifold.

For $r=\lambda_c$ such a contact structure $\xi'$ has to be overtwisted since in this case $\lambda_c$ represents the boundary of an overtwisted disk in the surgered manifold. For the existence of such an overtwisted contact structure see~\cite[pages 586--587]{DG1}. 

For $r\neq\lambda_c$ it follows that $\varphi^{-1}(\lambda_c)\neq\mu_0$. The classification of tight contact structures on solid tori with convex boundaries with two parallel dividing curves by Honda~\cite{H} implies that such a contact structure always exists and can be chosen to be tight. However, in general, this contact structure is not unique.

But if $r=\mu+q\lambda_c$ we can choose $\varphi\colon\lambda_0\mapsto\lambda_c$, for an arbitrary longitude $\lambda_0$ of the newly glued-in solid torus and the same classification result of Honda~\cite{H} says that there is only one tight contact structure with convex boundary and two dividing curves parallel to $\lambda_0$.
\end{proof}

\begin{ex}[A unique contact $(+2)$-surgery]\label{ex:unique}
Consider the contact surgery along the Legendrian unknot with $\operatorname{tb}=-1$ and (contact) surgery coefficient $+2$ (see Figure~\ref{fig:exoticsurgery}). The contact surgery coefficient $+2$ (measured with respect to the contact longitude $\lambda_c$) corresponds then to the slope $r=2\mu+\lambda_c$. By expressing this slope with respect to the surface longitude $\lambda_s$ one gets $r=\mu+\lambda_s$. It follows that the resulting manifold is by Example~\ref{ex:surgeryUnknot} topologically again $S^3$.

\begin{figure}[htbp] 
\centering
\def\svgwidth{0.99\columnwidth}
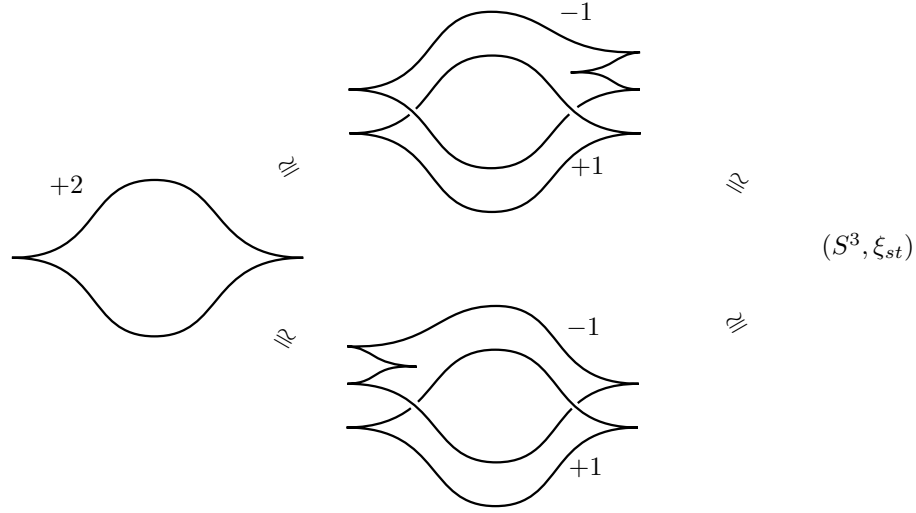
\caption{A unique contact $(+2)$-surgery resulting again in $(S^3,\xi_{st})$}
\label{fig:exoticsurgery}
\end{figure}	

Next, we want to show that the resulting contact structure is unique and leads again to $\xi_{st}$ (if one requires the contact structure on the new glued-in solid torus to be tight). In general, a contact Dehn surgery with contact surgery coefficient not of the form $1/q$ is not unique. But actually, in this example, it is. To see this, one first uses the algorithm in~\cite[Section~1]{DGS} (see also~\cite{DG1}) to change the contact surgery diagram into contact surgeries along a link with only $\pm1$ surgery coefficients (see Figure~\ref{fig:exoticsurgery}). Observe that different choices of stabilizations lead in general to different contact structures (and correspond exactly to the different contact structures on the glued-in solid torus), but in this case, the resulting contact structures are contactomorphic. The contactomorphism of the resulting manifold is induced by the contactomorphism $(x,y,z)\mapsto(-x,-y,z)$ of the old $(S^3,\xi_{st})$, that maps one link to the other (see also~\cite[~Section~9]{DG10}).

To see that this contact structure is really $\xi_{st}$, it is enough to show that the contact structure is symplectically fillable. In~\cite[Lemma~4.2.]{DGS} it is shown, that contact $(+1)$-Dehn surgery along the Legendrian unknot with $\tb=-1$ leads to $S^1\times S^2$ with the unique Stein fillable contact structure on it. Because contact $(-1)$-Dehn surgery along a Legendrian knot in a contact $3$-manifold $(M,\xi=\ker\alpha)$ corresponds to a symplectic $4$-handle attachment to the positive boundary of the symplectization $([0,1]\times M,d(e^t\alpha))$, it preserves symplectic fillability and the claim follows.
\end{ex}

To prove Theorem~\ref{thm:contact1} we now want to give the precise statement of Theorem~\ref{main:main0}, the generalization of Theorem~\ref{thm:gordon-luecke2} to the contact setting. For stating this result we introduce the Legendrian knots $U_n$ as Legendrian unknots with classical invariants $\tb(U_n)=-n$ and $\vert\rot(U_n)\vert= n-1$. 

The classification of Legendrian unknots in $(S^3,\xi_{st})$ by Eliashberg--Fraser~\cite{EF} says that two Legendrian unknots are equivalent (as oriented knots) if and only if they have the same $\tb$ and $\rot$ where the orientation of the knot is given by the sign of $\rot$. 

Here we want to prove that two Legendrian knots are equivalent if and only if their exteriors are contactomorphic. Since a knot exterior cannot determine the orientation of the knot this result can only hold for equivalence of unoriented knots. This is the reason why we consider Legendrian knots up to equivalence of unoriented knots.

From the theorem of Eliashberg--Fraser~\cite{EF} it follows that the knots $U_n$ are unique up to equivalence (of unoriented Legendrian knots). A front projection of a Legendrian unknot of type $U_n$ is shown in Figure~\ref{fig:legendreunknoten}.

\begin{figure}[htbp] 
\centering
\def\svgwidth{0.4\columnwidth}
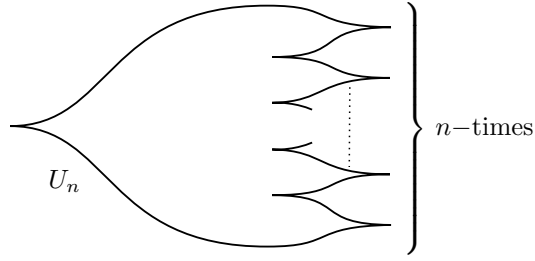
\raisebox{9.9ex}{$\left.\rule[9.1ex]{0pt}{\baselineskip}\right\}\,n-$times}
\caption{The front projection of a Legendrian unknot $U_n$}
\label{fig:legendreunknoten}
\end{figure} 

The generalization of Theorem~\ref{thm:gordon-luecke2} is then as follows.

	\begin{thm} [Contact Dehn surgery theorem]\label{thm:contact2}
Let $K$ be a Legendrian knot in $(S^3,\xist)$. If some result $(S^3_K(r),\xi_K(r))$ of contact $r$-surgery along $K$ is contactomorphic to $(S^3,\xist)$ for some $r\neq \mu$ then $K$ is equivalent to a Legendrian unknot $U_n$ with $\tb(U_n)=-n$ and $\rot(U_n)=\vert n-1\vert$.
\end{thm}

\begin{rem}[Non-uniqueness of the contact structure]
For general slopes the contact structure $\xi_K(r)$ is not unique. So one should read Theorem~\ref{thm:contact2} as follows: If there is a contact structure $\xi_K(r)$ on $S^3_K(r)$ such that $(S^3_K(r),\xi_K(r))$ is obtained from $(M,\xi)$ by contact Dehn surgery along $K$ with slope $r$ and if $(S^3_K(r),\xi_K(r))$ is contactomorphic to $(S^3,\xi_{st})$ for $r\neq \mu$, then the conclusion holds.
\end{rem}

The proof of Theorem~\ref{thm:contact2} is given in Section~\ref{section:appendix}. Assuming Theorem~\ref{thm:contact2} the proof of Theorem~\ref{thm:contact1} is now easy and similar to the topological case.

\begin{proof}[Proof of Theorem~\ref{thm:contact1}]\hfill\\
Pick a contactomorphism 
\begin{equation*}
h\colon\big(S^3\setminus\mathring{\nu K_1},\xist\big)\longrightarrow \big(S^3\setminus\mathring{\nu K_2},\xist\big).
\end{equation*}
And then consider again the following diagram:\\
\hfill\\
\begin{xy}
(0,40)*+{\big(S^3,\xist\big)}="a";(8,40)*+{\cong}="a1"; (24,40)*+{\big(S^3_{K_1}  (\mu_1),\xi_{K_1}  (\mu_1)\big)}="b"; (40,40)*+{:=}="b1"; (52,40)*+{\big(S^1\times D^2,\xi'\big)}="c"; (75,40)*+{+}="d"; (100,40)*+{\big(S^3\setminus\mathring{\nu K_1},\xist\big)}="e";(114,40)*+{\big/_\sim}="e1";%
(53,35)*+{\mu_0}="f"; (99,35)*+{\mu_1}="g";%
(75,20)*+{\circlearrowright}="k1";%
(53,5)*+{\mu_0}="l"; (92,5)*+{r_2:=h(\mu_1)}="m";%
(24,0)*+{\big(S^3_{K_2}  (r_2),\xi_{K_2}  (r_2)\big)}="n";(40,0)*+{:=}="n1"; (52,0)*+{\big(S^1\times D^2,\xi'\big)}="o"; (75,0)*+{+}="p"; (100,0)*+{\big(S^3\setminus\mathring{\nu K_2},\xist\big)}="q";(114,0)*+{\big/_\sim}="q1";%
{\ar@{-->}_{f} "b";"n"};{\ar@{->}@/_1pc/ _{\operatorname{Id}} "c";"o"};{\ar@{->}@/^1pc/^h "e";"q"};%
{\ar@{|->}_{\varphi_1} "f";"g"};{\ar@{|->}^{h\circ\varphi_1} "l";"m"};
\end{xy}\\
\hfill\\
Here the contact structure $\xi'$ denotes the unique tight contact structure on $S^1\times D^2$ with convex boundary corresponding to the slope $\mu_1$ (see proof of Theorem~\ref{thm:contsurg}). Because the contactomorphisms $\operatorname{Id}$ and $h$ on the two factors agree on the boundary convex surfaces, which determine the germ of the contact structures, these two maps glue together to a contactomorphism $f$ of the whole contact manifolds. From Theorem~\ref{thm:contact2} it follows that $r_2$ is equal to $\mu_2$ or $K_2$ is equivalent to $U_n$ for some $n$.

If $r_2=\mu_2$ then this is the trivial contact Dehn surgery, and so the contactomorphism $f$ maps $K_1$ to $K_2$.

In the other case, one makes the same argument with $K_1$ and $K_2$ reversed and concludes that $K_1$ is equivalent to $U_m$ for some $m$. Again the classification result of Eliashberg-Fraser~\cite{EF} implies that $K_1\sim U_m$ is equivalent to $U_n \sim K_2$ if and only if $n=m$. To show the last statement one observes that $(S^3\setminus\mathring{\nu K_1},\xist)$ is a solid torus with tight contact structure and convex boundary. Therefore, one can compute $-n=\operatorname{tb}(U_n)$ also as half the number of intersection points of the Seifert disk of $U_n$ with the dividing set of the convex boundary (see~\cite[Theorem~2.30]{Et}). Because the Seifert disks of $U_n$ and $U_m$ are both given by the $D^2$-factors of the exterior solid tori and because the exteriors are contactomorphic, the number of intersection points stays the same.
\end{proof}


\section{Computing the Thurston--Bennequin Invariant of a Legendrian knot in a surgery diagram}
\label{section:LOSS}

To prove Theorem~\ref{thm:contact2} one first determines with Theorem~\ref{thm:gordon-luecke2} all surgeries that lead again to $S^3$. The main problem is then that there are always many different choices for extending the contact structure over the new glued-in solid torus. But a very simple proof can be given by finding some new Legendrian knots in the exteriors of the surgery knots, that violates the Bennequin inequality in the new contact manifold. For doing this we want to present in this section a formula for computing the Thurston--Bennequin invariant of a Legendrian knot in a surgered manifold. 

The main problem when doing this is that the Thurston--Bennequin invariant is only defined for nullhomologous (or rationally nullhomologous) knots and a surgery, in general, destroys this property. But as long as the resulting manifold is a homology sphere (or a rational homology sphere) every knot has to be nullhomologous (or rationally nullhomologous). In this case formulas for computing the new Thurston--Bennequin invariants out of the old ones and out of the algebraic surgery data are given in~\cite[Lemma~6.6]{LOSS},~\cite[Lemma~2]{GO} and~\cite[Lemma~6.4]{C}. However, these formulas only work contact $(\pm1)$-surgeries and since we are concerned with a general contact $r$-surgery these formulas cannot be used here. Therefore, we present now a new formula to compute the Thurston--Bennequin invariant of a Legendrian knot presented in a general rationally contact $r$-surgery diagram along Legendrian knots. Moreover, our formula works also in general contact manifolds, for contact surgeries along non-Legendrian knots and also simplifies the computations in contact $(\pm1)$-surgery diagrams.

Building up on this and using the formulas in~\cite{LOSS,GO,C} one can also get formulas for computing rotation numbers of Legendrian knots, self-linking numbers of transverse knots and the $\de_3$-invariant of the resulting contact manifold in general contact $(1/n)$-surgery diagrams along Legendrian knots~\cite{DK}. In~\cite{DKK,DK2} similar formulas are given for computing the classical invariants of Legendrian knots sitting on the page of a contact open book.


\subsection{Computing the homology class of a knot}\hfill  \\
First, we want to give an easy (and easy to check) condition (out of the algebraic surgery data) when such a knot is nullhomologous in the new manifold. If this is the case we secondly show how to compute out of this data the new Thurston--Bennequin invariant.

Let $L=L_1\sqcup\cdots\sqcup L_n\subset S^3$ be an oriented link (where the choice of orientation is not important). And let $M$ be the $3$-manifold obtained out of $S^3$ by Dehn surgery along $L$ with topological surgery coefficients $r_i=p_i/q_i$, for $i=1,\ldots, n$. Denote by $L_0\subset S^3\setminus \mathring{\nu L}$ an oriented knot in $S^3$ and $M$ depending on the context. For simplicity write the linking numbers as $l_{ij}:=\lk(L_i,L_j)$, for $i=0,\ldots,n$. Set also
\begin{align*}
Q:=\begin{pmatrix}
p_1&q_2 l_{12} &\cdots&q_n l_{1n}\\
q_1 l_{21} & p_2&&\\
\vdots&&\ddots\\
q_1 l_{n1}&&& p_n
\end{pmatrix}\,\text{ and   }\,\,\, \mathbf{l}:=\begin{pmatrix}
l_{01}\\
\vdots\\
l_{0n}
\end{pmatrix}.
\end{align*}
The matrix $Q$ is a generalization of the linking matrix because for $q_i=1$ the matrix $Q$ is the linking matrix.

The knot $L_0$ is called \textbf{nullhomologous} in $M$ if $[L_0]=0\in H_1(M;\Z)$. One can show (see for example~\cite[Page~123]{GS}) that this is equivalent to the existence of a Seifert surface for the knot, so the surface longitude for a knot is defined if and only if the knot is nullhomologous. Recall also that the surface longitude $\lambda_s$ is independent of the choice of the explicit Seifert surface of the knot.

With the following lemma, one can decide from the algebraic surgery data if such a knot is nullhomologous in the surgered manifold.
\begin{lem}  [Nullhomologous knots]\label{lem:homology}
$L_0$ is nullhomologous in $M$ if and only if there exists an $\mathbf{a}\in\Z^n$ such that $\mathbf{l}=Q\mathbf{a}$.
\end{lem}

\begin{proof}  \hfill  \\
It is easy to compute the homology of $M$ (see for example~\cite[Proposition~5.3.11]{GS}) as
\begin{align*}
H_1(M;\mathbb{Z})=\mathbb{Z}_{\mu_1}\oplus\cdots\oplus\mathbb{Z}_{\mu_n}/\langle p_i\mu_i+q_i\sum_{\substack{j=1 \\ j\neq i}}^n l_{ij}\mu_j=0|i=1,\ldots,n\rangle,
\end{align*}
where the generators of the $\mathbb{Z}$-factors are given by right-handed meridians $\mu_i$ corresponding to the components $L_i$. Next, we express $L_0$ as a linear combination of the $\mu_i$. One can show that the coefficients are the linking numbers $l_{i0}$, i.e.
\begin{align*}
[L_0]=\sum_{i=1 }^n l_{i0}\mu_i.
\end{align*}
So $L_0$ is nullhomologous if and only if one can express $[L_0]=\sum l_{i0}\mu_i$ as a linear combination of the relations, i.e. if there exists integers $a_i$, $i=1,\ldots, n$, such that
\begin{align*}
\sum_{i=1 }^n l_{i0}\mu_i= \sum_{i=1}^n a_i  \big(p_i\mu_i+q_i\sum_{\substack{j=1 \\ j\neq i}}^n l_{ij}\mu_j\big)= \sum_{i=1}^n \big(a_i  p_i+\sum_{\substack{j=1 \\ j\neq i}}^n q_j l_{ij}a_j\big)\mu_i.
\end{align*}
By comparing the coefficients and by writing the corresponding equations in vector form one sees that this is true if and only if there exists a vector $\mathbf{a}\in\mathbb{Z}^n$ such that $\mathbf{l}=Q\mathbf{a}$.
\end{proof}


\subsection{Computing the Thurston--Bennequin invariant of a Legendrian knot}
\label{subsection:computtingtb}
Now assume $L_0$ is a Legendrian knot in $(S^3\setminus \mathring{\nu L},\xi_{st})\subset(S^3,\xi_{st})$. And let $\xi$ be a contact structure on $M$ that coincides with $\xi_{st}$ outside a tubular neighborhood of $L$. For example if $L$ is also a Legendrian link in $(S^3,\xi_{st})$ and $(M,\xi)$ is the result of a contact Dehn surgery along $L$. But it is important to notice that the setting here is a more general one, we can also use contact surgery along transverse knots or surgery along a knot that is not adapted to the contact structure. 

Here all surgery coefficients are understood to be topological surgery coefficients, i.e. with respect to the surface longitude $\lambda_s$ which has linking number zero with the knot. 

So if one has a Legendrian surgery diagram one first has to change the contact surgery coefficients to topological surgery coefficients, for example with the earlier mentioned formula 
\begin{equation*}
r_{i,top}=r_{i,cont}+\operatorname{tb}(L_i).
\end{equation*}

\begin{lem}  [Computing the Thurston--Bennequin invariant]\label{lem:tb}
If $L_0$ is nullhomologous in $M$, then one can compute the new Thurston--Bennequin invariant $\operatorname{tb}_{new}$ of $L_0$ in $(M,\xi)$ from the old one $\operatorname{tb}_{old}$ of $L_0$ in $(S^3,\xi_{st})$ as
\begin{align*}
\operatorname{tb}_{new}=\operatorname{tb}_{old}-\sum_{i=1}^n{a_i q_i l_{i0}}
\end{align*}
where $\mathbf{a}$ is a vector given by the formula from Lemma~\ref{lem:homology}.
\end{lem}

\begin{proof}  \hfill  \\
Let $\lambda_s$ be the surface longitude of $L_0$ in $S^3$, i.e. $\operatorname{lk}(L_0,\lambda_s)=0$. And let $\lambda_c$ be the contact longitude of $L_0$ in $(S^3,\xi_{st})$. Then $\operatorname{tb}_{old}$ is given by
\begin{align*}
\lambda_c=\operatorname{tb}_{old}\mu_0+\lambda_s \in H_1(\partial \nu L_0).
\end{align*}
Because the contact longitude is defined by the contact structure along $L_0$ and the contact structure does not change near $L_0$ by doing the surgery along $L$ the contact longitude $\lambda_c$ represents also the contact longitude in $(M,\xi)$. But in general the surface longitude $\lambda_s$ changes. Since the knot $L_0$ is nullhomologous in $M$, there is a unique $f\in\mathbb{Z}$ such that $f\mu_0+\lambda_s=0\in H_1(M\setminus \mathring{\nu L_0};\mathbb{Z})$ (this is the new surface longitude). Then $\operatorname{tb}_{new}$ is given by
\begin{align*}
\lambda_c=\operatorname{tb}_{new}\mu_0+(f\mu_0+\lambda_s) \in H_1(\partial \nu L_0).
\end{align*}
Putting this together leads to
\begin{align*}
\operatorname{tb}_{new}=\operatorname{tb}_{old}-f.
\end{align*}
So the only thing left is to compute $f$. As in the proof of Lemma~\ref{lem:homology} one computes the homology of $M\setminus \mathring{\nu L_0}$ as
\begin{align*}
H_1(M\setminus \mathring{\nu L_0};\mathbb{Z})=\mathbb{Z}_{\mu_0}\oplus\cdots\oplus\mathbb{Z}_{\mu_n}/\langle p_i\mu_i+q_i\sum_{\substack{j=0 \\ j\neq i}}^n l_{ij}\mu_j=0|i=1,\ldots,n\rangle
\end{align*}
and expresses $\lambda_s$ as
\begin{align*}
\lambda_s=\sum_{i=1 }^n l_{i0}\mu_i.
\end{align*}
So $f\mu_0+\lambda_s$ is zero in $H_1(M\setminus \mathring{\nu L_0};\mathbb{Z})$ if and only if there exists integers $b_i\in\mathbb{Z}$, $i=1,\ldots,n$, such that
\begin{align*}
f\mu_0+\sum_{i=1 }^n l_{i0}\mu_i&= \sum_{i=1}^n b_i  \big(p_i\mu_i+q_i\sum_{\substack{j=0 \\ j\neq i}}^n l_{ij}\mu_j\big)\\
&= \big(\sum_{j=1 }^n b_j q_j l_{j0}\big) \mu_0 + \sum_{i=1}^n \big(b_i  p_i+\sum_{\substack{j=1 \\ j\neq i}}^n q_j l_{ij}b_j\big)\mu_i.
\end{align*}
This is equivalent to the existence of a vector $\mathbf{b}\in\mathbb{Z}^n$ such that
\begin{align*}
\mathbf{l}&=Q\mathbf{b}\,\,\text{and}\\
f&=\sum_{j=1 }^n b_j q_j l_{j0}.
\end{align*}
If one chooses for $\mathbf{b}$ a solution $\mathbf{a}$ from Lemma~\ref{lem:homology} the formula follows.
\end{proof}

\begin{ex} [Computing  $\operatorname{tb}$ with this formulas]  \label{ex:proof}
(1) Consider the surgery diagram from Figure~\ref{fig:computing-tb-ex1} (i). The old Thurston--Bennequin invariant of $L_0$ is $-1$. And because the surgery along the unknot $L_1$ with topological framing $1/n$ leads again to $S^3$ the knot $L_0$ is again nullhomologous in the new manifold. This can be checked also with the formulas from Lemma~\ref{lem:homology} and~\ref{lem:tb}. For $\mathbf{a}$ one gets the equation $1=l_{10}=Qa_1=p_1 a_1=a_1$. And therefore $tb_{new}=-1-l_{10}q_1a_1=-1-n$. Observe again that this result does not depend on the explicit contact structure chosen for the surgery.\\
(2) Consider the surgery diagram from Figure~\ref{fig:computing-tb-ex1} (ii). Again the surgery leads to $S^3$ so the resulting knot is again nullhomologous. The new Thurston--Bennequin invariant can be computed as follows.
For $\mathbf{a}$ one gets the equation $2=l_{10}=Qa_1=p_1 a_1=a_1$. And therefore $tb_{new}=-1-l_{10}q_1a_1=-1-4n$.

In the next section, we will see that Theorem~\ref{thm:contact2} is an easy corollary out of these two examples.\\
(3) There are also examples where the solution $a$ of $\mathbf{l}=M\mathbf{a}$ is not unique. But the result of $tb_{new}$, of course, is not affected by this. For example, consider the surgery diagram from Figure~\ref{fig:computing-tb-ex3} (i). First one has to check if the knot $L_0$ is again nullhomologous in the surgered manifold. For this, one has to look if there exists a solution $\mathbf{a}\in\mathbb{Z}$ of 
\begin{align*}
\begin{pmatrix}
 1\\
2
 \end{pmatrix}=\mathbf{l}=Q\mathbf{a}=\begin{pmatrix}
 1&1\\
2&2
 \end{pmatrix}\begin{pmatrix}
 a_1\\
a_2
 \end{pmatrix}.
\end{align*}
Two obvious solutions are
\begin{align*}
\mathbf{a}=\begin{pmatrix}
 1\\
0
 \end{pmatrix}\,\,\text{or }\,\,\mathbf{a}=\begin{pmatrix}
 0\\
1
 \end{pmatrix}.
\end{align*}
So the new knot is nullhomologous and the new Thurston--Bennequin invariant can be computed out of the old one as follows
\begin{align*}
tb_{new}=tb_{old}-a_1q_1l_{10}-a_2q_2l_{20}=tb_{old}-2a_1-2a_2=tb_{old}-2.
\end{align*}
Observe that this result does not depend on the choice of $\mathbf{a}$.
\end{ex}

\begin{figure}[htbp] 
\centering
\def\svgwidth{0.9\columnwidth}
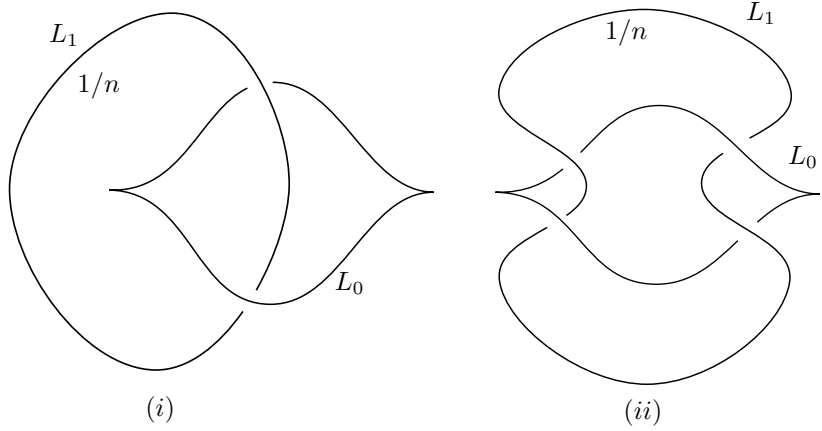
\caption{Computing $\operatorname{tb}$-invariants in surgery diagrams}
\label{fig:computing-tb-ex1}
\end{figure}


\subsection{Rationally Nullhomologous Knots}\hfill  \\
\label{subsection:computtingrationallytb}
These results can be easily generalized to rationally nullhomologous knots. A knot $L_0$ in $M$ is called \textbf{rationally nullhomologous} if there exists a natural number $k\in\mathbb{N}$ such that $k[L_0]=0\in H_1(M;\mathbb{Z})$. For rationally nullhomologous Legendrian knots in contact $3$-manifolds one can generalize the Thurston--Bennequin invariant to the so-called \textbf{rational Thurston--Bennequin invariant} $\operatorname{tb}_\mathbb{Q}$ (see for example~\cite[Definition~6.2]{BG} or~\cite[Section~2~and~3]{GO}, for the fact that this is in general well defined see~\cite[Section~5]{DKK}). 

With the same notation as from the first sections, one gets the following generalizations of these results.

\begin{lem} [Computing the rationally Thurston--Bennequin invariant] 
(1) $L_0$ is rationally nullhomologous in $M$ if and only if there exists a natural number $k\in\mathbb{N}$ and a vector $\mathbf{a}\in\mathbb{Z}^n$ such that $k\mathbf{l}=Q\mathbf{a}$.\\
(2) If $L_0$ is rationally nullhomologous in $M$ then one can compute the new rational Thurston--Bennequin invariant $\operatorname{tb}_{\mathbb{Q}, new}$ of $L_0$ in $(M,\xi)$ from the old one $\operatorname{tb}_{old}$ of $L_0$ in $(S^3,\xi_{st})$ as follows
\begin{align*}
\operatorname{tb}_{\mathbb{Q}, new}=\operatorname{tb}_{old}-\frac{1}{k}\sum_{i=1}^n{a_i q_i l_{i0}}.
\end{align*}
\end{lem}

\begin{proof}  \hfill  \\
The proofs are similar to the ones in the foregoing subsections. For the first part one has to change the condition $[L_0]=0$ to $k[L_0]=0$ and do the same computations again.

The second part works similarly. If $L_0$ is only rationally nullhomologous in $M$, then the Thurston--Bennequin invariants are related as follows (see also~\cite[Proof of Lemma~2]{GO}).
\begin{align*}
k\lambda_s+k\operatorname{tb}_{old}\mu_0=k\lambda_c=k\operatorname{tb}_{\mathbb{Q}, new}\mu_0+(f\mu_0+k\lambda_s).
\end{align*}
Then same computations as in the first part lead to the result.
\end{proof}

\begin{ex} [Computing $\operatorname{tb}$ of rationally Legendrian unknots in lens spaces]
This formula for computing the rational Thurston--Bennequin invariant is very useful to calculate $\operatorname{tb}_\mathbb{Q}$ in lens spaces. For example, consider the surgery diagram from Figure~\ref{fig:computing-tb-ex3} (ii). The $(-p/q)$-surgery along $L_1$ leads to the lens space $L(p,q)$. To check if the knot $L_0$ is nullhomologous one has to solve the equation $1=\mathbf{l}=Q\mathbf{a}=-pa_1$. For $p\neq1$ this equation has no solution in $\mathbb{Z}$ and therefore $L_0$ is not nullhomologous in the surgered manifold. But for $p\neq0$ the equation $k=k\mathbf{l}=Q\mathbf{a}=-pa_1$ has a solution, for example $k=p$ and $a_1=-1$. So $L_0$ is rationally nullhomologous in $L(p,q)$. The rationally Thurston--Bennequin invariant is computed as follows
\begin{align*}
tb_{\mathbb{Q},new}=tb_{old}-\frac{1}{k}a_1l_{01}q_1=tb_{old}+\frac{q}{p}.
\end{align*}
Observe again that the result is independent of the chosen solution $\mathbf{a}$ and independent of the chosen contact structure on the new glued-in solid torus. 
\begin{figure}[htbp] 
\centering
\def\svgwidth{0.7\columnwidth}
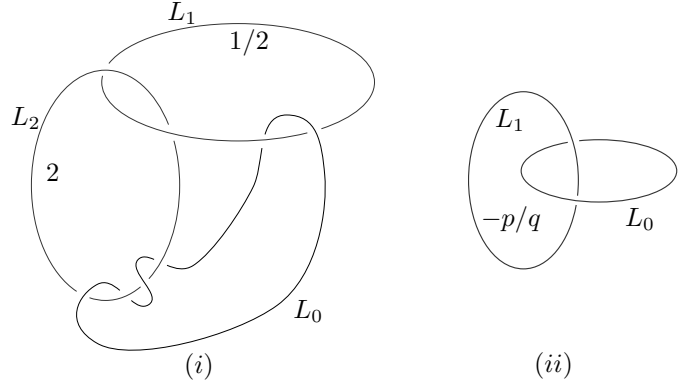
\caption{Computing rationally $\operatorname{tb}$-invariants in surgery diagrams}
\label{fig:computing-tb-ex3}
\end{figure} 
\end{ex}


\subsection{Extension to surgeries on general manifolds}\hfill  \\
One can also study the same problem for a surgery in a general contact manifold (not on $(S^3,\xi_{st})$). This is motivated by~\cite[Lemma~6.4]{C}.

Consider now $L=L_1\sqcup\cdots\sqcup L_n\subset N^3$ an oriented nullhomologous link in some contact $3$-manifold $(N,\xi_N)$. Denote by $(M,\xi)$ some result of contact surgery along $L$ and by $L_0\subset (N\setminus \mathring{\nu L},\xi_N)$ an oriented nullhomologous Legendrian knot in $(N,\xi_N)$ and $(M,\xi)$, depending on the context. Then one gets exactly the same formulas as before. The only part changing in the proof is that the homologies are different, for example
\begin{align*}
H_1(M;\mathbb{Z})=H_1(N;\mathbb{Z})\oplus\mathbb{Z}_{\mu_1}\oplus\cdots\oplus\mathbb{Z}_{\mu_n}/\langle p_i\mu_i+q_i\sum_{\substack{j\neq i}} l_{ij}\mu_j=0|i=1,\ldots,n\rangle.
\end{align*}
For more details see~\cite[Proof of Lemma~6.4]{C}.

If one has a surgery diagram with also $1$-handles included, then one can use the above methods as well. The first possibility is to change all $1$-handles into topological $0$-surgeries along unknots and the second possibility is to think of the surgery diagram as a surgery diagram in $(\#_n S^1\times S^2,\xi_{st})$ (represented by $n$ $1$-handles) and then use the above extension.


\section{Proof of Theorem \ref{thm:contact2}}	
\label{section:appendix} 

Let $K$ be a Legendrian knot in $(S^3,\xi_{st})$ such that $(S^3_K(r),\xi_K(r))$ is again a contact $S^3$ (with any contact structure) for some $r\neq \mu$. From Theorem~\ref{thm:gordon-luecke2} it follows, that $K$ is topologically equivalent to an unknot $U$. (So we will write $U$ instead of $K$.) In Example~\ref{ex:surgeryUnknot} and Remark~\ref{rem:hom} we explained that the topological surgery coefficient (with respect to the surface longitude $\lambda_s$) has to be of the form $1/q$, for $q\in\Z$. 

Now we want to show that every resulting contact structure $\xi_U(r)$ is overtwisted if $U$ is not coarse equivalent to an unknot of the form $U_n$. To do this we will show that in the resulting contact $3$-spheres there exist Legendrian knots that cannot be realized in $\xist$. 

For this, one considers Examples~\ref{ex:proof} (1) and (2). By doing a Rolfsen twist along $U$ one sees that the Legendrian knot $L_0$ from Examples~\ref{ex:proof} (1) remains an unknot in the new surgered manifold. For $q<0$ one gets $\operatorname{tb}_{new}=-q-1>0$ (independent of the choice of the contact structure on the new glued-in solid torus). According to the Bennequin inequality, this knot cannot lie in a tight contact structure. So for $q<0$ it is not possible to have a non-trivial contact surgery from $(S^3,\xi_{st})$ to itself.
\begin{figure}[htbp] 
\centering
\def\svgwidth{0.9\columnwidth}
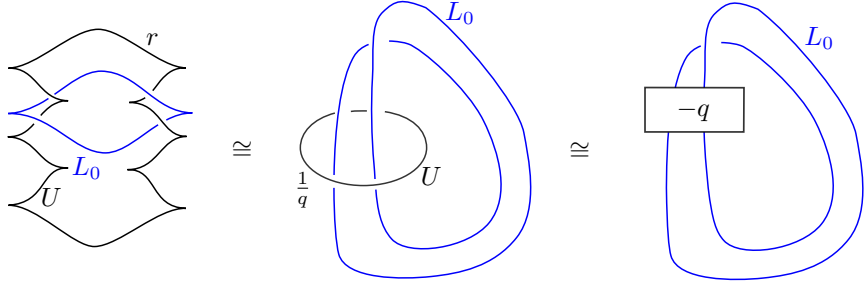
\caption{An unknot $L_0$ that becomes a negative $(2,2q+1)$-torus knot after surgery}
\label{fig:proofofcontactsurgerytheorem21}
\end{figure}

For $q>0$ one looks at Examples~\ref{ex:proof} (2). In Figure~\ref{fig:proofofcontactsurgerytheorem21} it is shown (again by doing a Rolfsen twist) that $L_0$ becomes in the new surgered manifold a negative $(2,2q+1)$-torus knot $T_{2,2q+1}$. In~\cite[Section~1]{EH} it is proven that the maximal Thurston--Bennequin invariant of such a knot in $(S^3,\xi_{st})$ is given by $-2-4q$, which is smaller than $\operatorname{tb}_{new}=-1-4q$. For $U$ not coarse equivalent to a Legendrian unknot of the form $U_n$ this example can be realized as a contact surgery along a Legendrian knot, which proves the result.
\hfill$\square$


\section{A contact Rolfsen twist and counterexamples to the Legendrian link exterior problem}	
\label{counterexample}

The next natural question is if there exist (like in the topological case) Legendrian links not determined by the contactomorphism type of their exteriors. If one wants to generalize Example~\ref{ex:Whitehead} one needs a contact analogon of a Rolfsen twist. In Example~\ref{ex:unique} we gave a contact surgery from $(S^3,\xi_{st})$ to $(S^3,\xi_{st})$. Topologically this surgery represents a $(+1)$-Dehn surgery along an unknot. Deleting such components from contact surgery diagrams leads to a contact analogon of a $(-1)$-Rolfsen twist. But it is not clear how the rest of the diagram changes then.

\begin{lem}[A contact Rolfsen twist]\label{lem:rolfsentwist}
$(a)$ The two contact surgery diagrams shown in the upper row of Figure~\ref{fig:rolfsentwist} represent contactomorphic contact manifolds. \\
$(b)$ The same is true for the two surgery diagrams in the lower row of Figure~\ref{fig:rolfsentwist}.\\
$(c)$ It follows, that the non-unique contact surgery diagram on the left of Figure~\ref{fig:rolfsentwist} is contactomorphic to exactly one contact surgery diagram on the right of the same figure.
\end{lem}

\begin{figure}[htbp] 
\centering
\def\svgwidth{0.98\columnwidth}
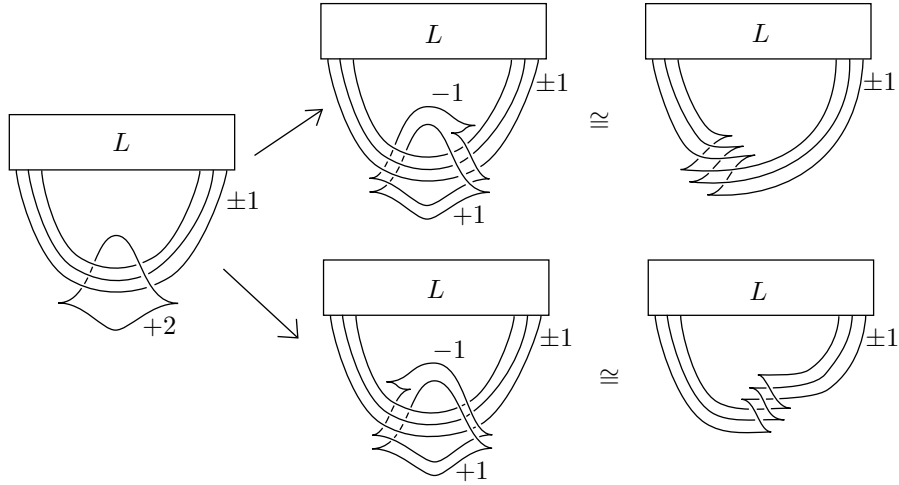
\caption{A contact Rolfsen twist. The box $L$ represents an arbitrary Legendrian link $L$ with all contact surgery coefficients equal to $\pm1$.}
\label{fig:rolfsentwist}
\end{figure} 

\begin{rem}[Doing $n$-fold contact Rolfsen twists]
Of course one can do such a contact Rolfsen twist more than once, this corresponds to doing $(+2)$-contact Dehn surgeries along $n$ disjoint copies of the Legendrian unknot. By translating the contact surgery diagram into an open book as in the following proof one can show that this is the same as doing a single contact Dehn surgery along the unknot with contact surgery coefficient $1+\frac{1}{n}$.
\end{rem}

\begin{proof}[Proof of Lemma~\ref{lem:rolfsentwist}]\hfill\\
We prove part $(b)$. Part $(a)$ works similar and $(c)$ follows then together with Example~\ref{ex:unique}.

First, one observes that one can put the whole Legendrian link $L$ together with the Legendrian unknot $U$ on the pages of an abstract open book for $(S^3,\xi_{st})$. The rough idea is to choose a fine enough $CW$-decomposition of $S^3$ such that the $1$-skeleton is a Legendrian graph containing the Legendrian link. Then the positive and the negative transverse push-off of this Legendrian graph represents the binding of the open book. A page is given by a surface bounded by the transverse push-offs and containing the $1$-skeleton. Consequently, this open book decomposition contains $L$ in its pages (for details see for example~\cite[Section~2.3]{A}).
\begin{figure}[htbp] 
\centering
\def\svgwidth{0.85\columnwidth}
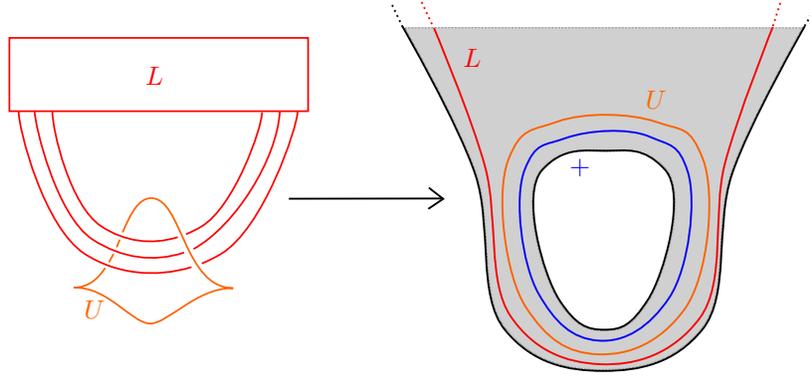
\caption{Putting $L$ and $U$ on an open book for $(S^3,\xi_{st})$}
\label{fig:rolfsenopenbook}
\end{figure} 

The method in~\cite{A} is very effective to construct an open book of $(S^3,\xi_{st})$ with a given Legendrian link on its pages. In the proof of Theorem~5.5 and Lemma~5.8. in~\cite{A} an open book for $(S^3,\xist)$ with the Legendrian knot $L$ together with the Legendrian unknot $U$ on its page is explicitly constructed. Here we are only interested in contact manifolds up to contactomorphism (rather than isotopy), therefore, it is enough to consider abstract open books (instead of embedded open books). A part of the open book (seen as an abstract open book) constructed in~\cite[Proof of Theorem~5.5 and Lemma~5.8.]{A} is shown in Figure~\ref{fig:rolfsenopenbook}. The page is pictured in gray and a part of the monodromy is described as a right-handed Dehn twist along the blue curve. One has to read the red curve in the open book, as necessarily many parallel copies.

Next, one constructs from this an abstract open book for the surgered manifold as explained in~\cite{O}. First, one puts the Legendrian link $L$ together with the Legendrian unknots $U_1$ and $U_2$ from the $(\pm1)$-surgery diagram on the page of an abstract open book for $(S^3,\xi_{st})$. This is shown in Figure~\ref{fig:rolfsenproofneu} on the upper right side. The additional stabilization of the Legendrian unknot $U_2$ corresponds to a stabilization of the open book as shown in~\cite[Figure 1 and 2]{O}.
\begin{figure}[htbp] 
\centering
\def\svgwidth{0.98\columnwidth}
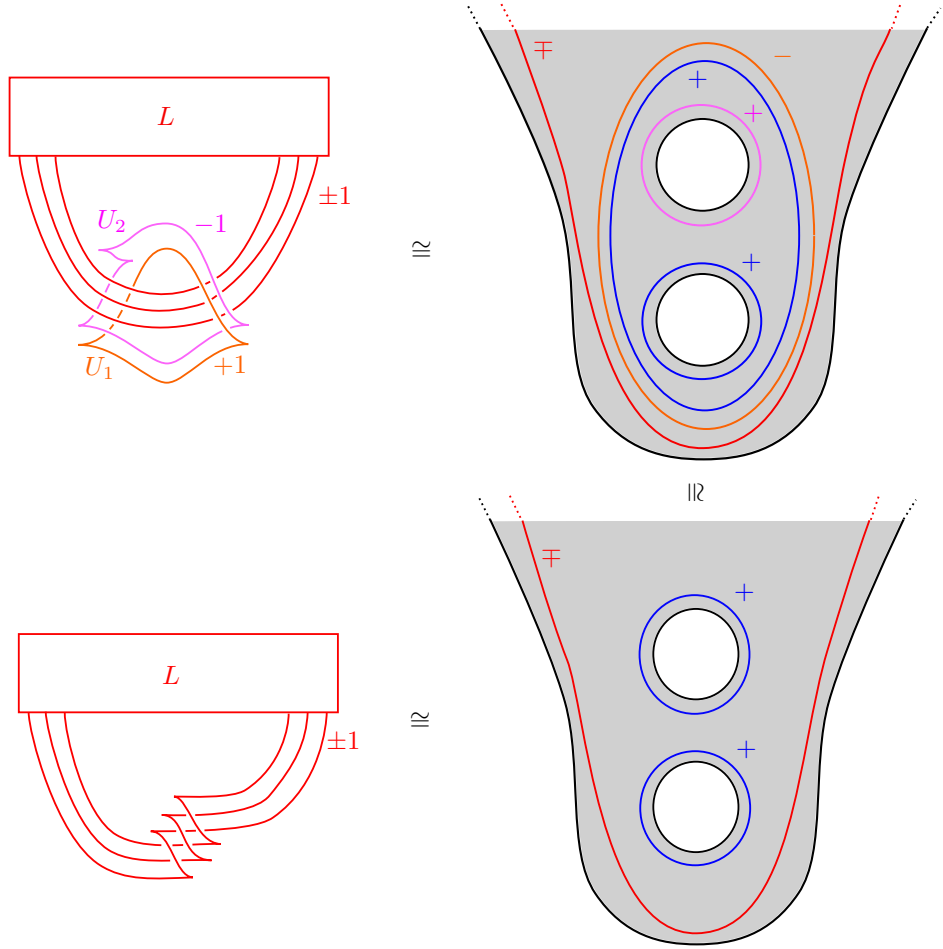
\caption{Proof of the contact Rolfsen twists via open book decompositions}
\label{fig:rolfsenproofneu}
\end{figure} 

The monodromy of the open book for $(S^3,\xi_{st})$ is given by right-handed Dehn twists along the blue curves. The monodromy of the surgered manifold is obtained by composing this old monodromy with right-handed Dehn twists along the $(-1)$-surgery knot and left-handed Dehn twists along the $(+1)$-surgery knot (see~\cite[Proposition 8]{O}).

Two Dehn twists cancel each other and one gets the open book in the bottom right corner of Figure~\ref{fig:rolfsenproofneu}. That open book represents the surgery diagram of the once stabilized Legendrian link $L$ shown in the bottom left corner of Figure~\ref{fig:rolfsenproofneu}.

Observe also that the topological surgery coefficients of $L$ change exactly as prescribed by the topological Rolfsen twists.
\end{proof}

\begin{ex} [Counterexamples to the Legendrian link exterior problem]\label{ex:counterlink}
With this contact Rolfsen twist, one can give counterexamples to the Legendrian link exterior problem. But it is not as easy as in the topological setting. The main point there was that the new glued-in solid tori is again a tubular neighborhood of the spine of this torus. This is not the case in the contact setting. 

To see this, consider a Legendrian knot $K$ in a contact $3$-manifold $(M,\xi)$ and the result of contact surgery along this knot $(M_K(r),\xi_K(r))$. Then the slope of the new glued-in solid torus is $r$. But if this new glued-in solid torus would be the standard neighborhood of a Legendrian knot then the slope would be of the form $\lambda+n\mu$. So for general $r$ this is not the case.

Therefore, one looks at the contact surgery diagram with one $(+1)$-surgery along $U_1$ and one $(-1)$-surgery along $U_2$. The new glued-in solid tori are again in a canonical way standard neighborhoods of the Legendrian spines $U_1'$ and $U_2'$. From the proof of Lemma~\ref{lem:rolfsentwist} it follows that this spines $U_1'$ and $U_2'$ lie in the resulting surgery diagram as shown in Figure~\ref{fig:rolfsentwistmitseele}.
\begin{figure}[htbp]  
\centering
\def\svgwidth{0.96\columnwidth}
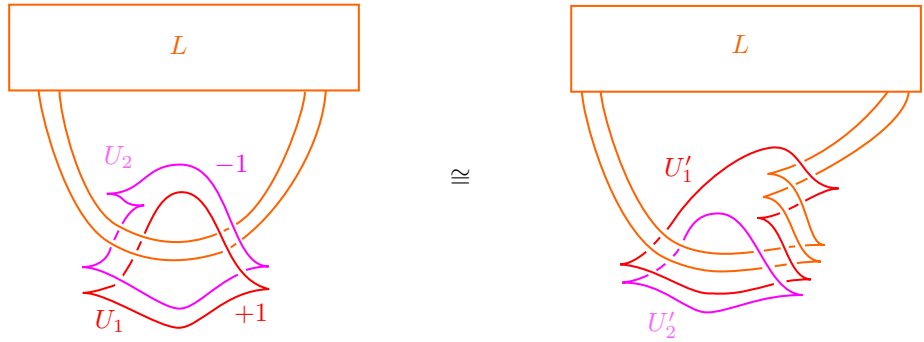
\caption{The spines of the new glued in solid tori}
\label{fig:rolfsentwistmitseele}
\end{figure} \hfill\\
(1) Consider the two Legendrian links $L\sqcup U_1\sqcup U_2$ and $L'\sqcup U_1'\sqcup U_2'$ in $(S^3,\xi_{st})$ as depicted in Figure~\ref{fig:counterexample1}. These two links are not equivalent because their triples of $\operatorname{tb}$-invariants are different: $\operatorname{tb}(L\sqcup U_1\sqcup U_2)=(-1,-1,-2)$; $\operatorname{tb}(L'\sqcup U_1'\sqcup U_2')=(-2,-2,-1)$. 

But their exteriors are contactomorphic, as one can see as follows. One does a $(+1)$-surgery along $U_1$ and a $(-1)$-surgery along $U_2$. The resulting manifold is again $(S^3,\xi_{st})$, in which the Legendrian knot $L$ looks now like $L'$ in Figure~\ref{fig:counterexample1} on the right. So in the exteriors of $U_1\sqcup U_2$ and $U_1'\sqcup U_2'$ the Legendrian knots $L$ and $L'$ are the same.
\begin{figure}[htbp] 
\centering
\def\svgwidth{0.96\columnwidth}
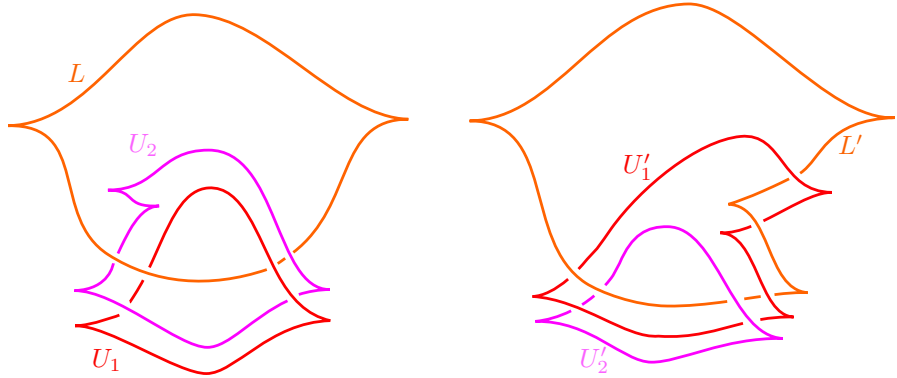
\caption{Two different Legendrian links with contactomorphic exteriors}
\label{fig:counterexample1}
\end{figure} \hfill\\
(2) One can also get examples with different topological types and the same Thurs\-ton--Bennequin invariants. For that consider the Legendrian links in Figure~\ref{fig:counterexample2} similar to the Whitehead links as in Example~\ref{ex:Whitehead}. In the left Legendrian link in Figure~\ref{fig:counterexample2} all three knots are unknots, but in the right link the knot $L'$ is non-trivial, so they cannot be equivalent. But their exteriors are contactomorphic with the same argument as in the foregoing example. 
\begin{figure}[htbp] 
\centering
\def\svgwidth{0.99\columnwidth}
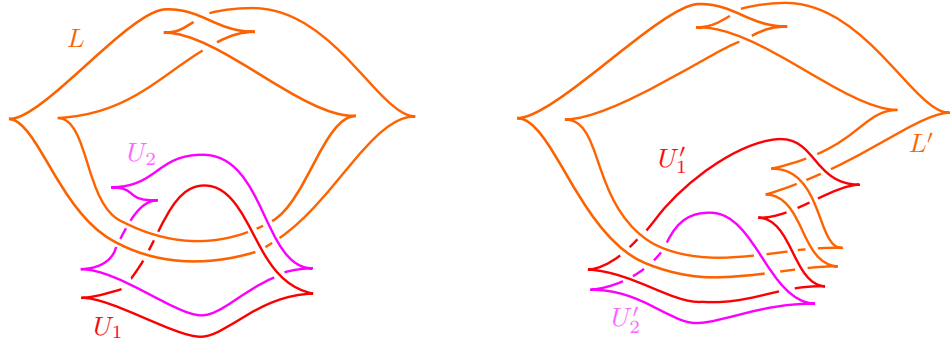
\caption{Two different Legendrian links with contactomorphic exteriors}
\label{fig:counterexample2}
\end{figure} \hfill\\
(3) By doing contact Dehn surgeries corresponding to an $n$-fold Rolfsen twist, for $n>0$, one gets in both foregoing examples infinitely many different pairs of Legendrian links such that each pair has contactomorphic exteriors. But it is not clear if there also exist infinitely many different Legendrian links whose exteriors are all contactomorphic to each other.
\end{ex}


	\section{The knot complement problem in general manifolds}	
	\label{topgenmfd} 
Instead of looking at Legendrian knots in $(S^3,\xi_{st})$ one can also look at Legendrian knots in general contact $3$-manifolds and study the Legendrian knot exterior problem for these knots. Before studying the contact case in the next section, we first recall some basic facts about the knot complement problem in general manifolds in the topological setting. 

The following lemma whose proof is similar as the one of Theorem~\ref{thm:gordon-luecke1} transfers the knot complement problem in an arbitrary manifold to a problem concerning Dehn surgery.

\begin{lem}[Criterion for the knot complement problem]\label{lem:crit}
Let $K$ be a knot in a $3$-manifold $M$, such that there is no non-trivial Dehn surgery along $K$ resulting again in $M$. Then the equivalence type of $K$ is determined by the diffeomorphism type of its complement.
\end{lem}

However, note that in general manifolds there exist knots $K$ such that there are non-trivial Dehn surgeries along $K$ not changing the manifold. Also, the equivalence type of a knot, in general, is not determined by the diffeomorphism type of its complement. To see this, consider the following example (see also~\cite{Ro}).

\begin{ex}  [Two non-equivalent knots with the same complements]\label{ex:countergeneral}
Consider two different Dehn surgeries along the unknot $U$ in $S^3$ with the (topological) surgery coefficients $r_1=-5/2$ and $r_2=-5/3$, leading to the lens spaces $L(5,2)$ and $L(5,3)$. By the classification of lens spaces~\cite[Exercise~5.3.8.~(b)]{GS} these two lens spaces are orientation preserving homeomorphic and the homeomorphism is given by interchanging the two solid tori (see also~\cite[Section~2]{GO}).

From this Dehn surgery example, it is easy to find two non-equivalent knots with the same exteriors. For this write  
\\
\hfill\\
\begin{xy}
(0,5)*+{L(5,2)}="a";(10,5)*+{\cong}="a1"; (24,5)*+{S^3_{U}  (r_i)}="b"; (39,5)*+{:=}="b1"; (52,5)*+{S^1\times D^2}="c"; (70,5)*+{+}="d"; (90,5)*+{S^3\setminus\mathring{\nu U}}="e";(104,5)*+{\big/_\sim}="e1";%
(52,0)*+{\mu_0}="f"; (89,0)*+{r_i}="g";%
{\ar@{|->} "f";"g"};
\end{xy}\\
\hfill\\
and consider the knots 
\begin{equation*}
K_i:=S^1\times \{0\}\subset S^1\times D^2  \subset L(5,2).
\end{equation*}
The knots $K_1$ and $K_2$ given as the spines of the new glued-in solid tori represents the spines of the genus-1 Heegaard splitting of $L(5,2)$. As tubular neighborhood of $K_i$ one chooses the whole new glued-in solid tori $S^1\times D^2$, therefore the exterior is in both cases $S^3\setminus\mathring{\nu U}$. It remains to show that these two knots are not equivalent. Therefore, assume that there is an orientation-preserving homeomorphism
\[\begin{array}{rccc}
f\colon& L(5,2)&\longrightarrow& L(5,2)\\
	&K_1 &\longmapsto& K_2.
\end{array}\]
By restricting $f$ to the complementary solid tori
\begin{equation*}
L(5,2)\setminus \mathring{\nu K_i}=S^3\setminus\mathring{\nu U}=T_2
\end{equation*}
one gets a homeomorphism
\[\begin{array}{ccc}
	 T_2&\longrightarrow& T_2\\
	r_1 &\longmapsto& r_2,
\end{array}\]
which sends the slope $r_1=-5\lambda_2+2\mu_2$ to the slope $r_2=-5\lambda_2+3\mu_2$. But such a map cannot exist because all orientation-preserving homeomorphisms of solid tori are isotopic to Dehn twists along meridians. So $K_1$ is not orientation preserving equivalent to $K_2$ in $L(5,2)$.

With the same methods as above, it is easy to show (see~\cite{Ro}) that if $K_1$ and $K_2$ are the cores of the two solid tori in the standard Heegaard splitting of $L(p,q)$, then they have homeomorphic complements, but there is an orientation-preserving (reversing) homeomorphism of $L(p,q)$ sending $K_1$ to $K_2$ if and only if $q^2\equiv 1 \, (\operatorname{mod}\,p)$ ($q^2\equiv -1 \, (\operatorname{mod}\,p)$).
\end{ex}

The key point in the foregoing example is that there is a so-called \textbf{exotic cosmetic surgery}, that means two surgeries along the same knot resulting in the same manifold but with different slopes, such that there is no homeomorphism of the knot exterior mapping one slope to the other (see~\cite{BHW}). One can show that every knot in a given $3$-manifold is determined by its complement if and only if this manifold cannot be obtained by exotic cosmetic surgery from another manifold.

\begin{thm}[Exotic cosmetic surgeries]\label{thm:cosmetic}
Let $K$ be a knot in a closed $3$-manifold $M$. The following two claims are equivalent.
\begin{enumerate}
	\item The equivalence type of $K$ in $M$ is determined by the homeomorphism type of its exterior $M\setminus \mathring{\nu K}$.
	\item For any knot $K'$ in any $3$-manifold $M'$ such that $M'_{K'}(r_1)$ and $M'_{K'}(r_2)$ are both homeomorphic to $M$ (for $r_1\neq r_2$) and $K$ is in both cases given as the spine of the newly glued-in solid torus there exists a homeomorphism of the knot exterior
	\[\begin{array}{rccc}
	 h\colon& M'\setminus\mathring{\nu K'}&\longrightarrow& M'\setminus\mathring{\nu K'}\\
	&r_1 &\longmapsto& r_2,
\end{array}\]
mapping one slope to the other.
\end{enumerate}
\end{thm}

\begin{rem}[Oriented exotic cosmetic surgeries]\label{rem:orientedcosmetic}
(1) The statement holds in the oriented and unoriented case.\\
(2) One can take the manifold $M'$ to be homeomorphic to $M$, this explains the name cosmetic surgery.
\end{rem}

\begin{proof} [Proof of Theorem~\ref{thm:cosmetic}] \hfill\\
The proof of $(1)\Rightarrow(2)$ works exactly as in Example~\ref{ex:countergeneral}. The implication $(2)\Rightarrow(1)$ is similar to the proof of Theorem~\ref{thm:gordon-luecke1}.
\end{proof}

So the study of the knot complement problem is equivalent to the study of exotic cosmetic surgeries. In the topological setting, some is known, but much remains open. Beside the discussed knot complement theorem for knots in $S^3$, Gabai showed in~\cite{Ga} that knots in $S^1\times S^2$ (or more generally in a connected sum of arbitrary $T^2$- or $S^2$-bundles over $S^1$) are determined by their complements. 

But in general manifolds, this will not hold. Building up on work by Mathieu~\cite{Ma}, Rong classified in~\cite{Ro} all knots in $3$-manifolds with Seifert fibered complements, that are not determined by their complements. These knots are given by the spines of the solid tori in the standard Heegaard splitting of some special lens spaces $L(p,q)$ as described in Example~\ref{ex:countergeneral} or as exceptional fibers of index $2$ in special Seifert fibered manifolds. But all these homeomorphisms sending one of the exceptional fibers to another one have to be orientation-reversing. 

Later Matignon~\cite{M} proved that all non-hyperbolic knots in atoroidal irreducible Seifert fibered $3$-manifolds are determined by their complements (except the cores of the standard Heegaard splittings in Lens spaces). 

For hyperbolic knots, there is until now only one counterexample. In~\cite{BHW} Bleiler, Hodgson and Weeks construct two non-equivalent hyperbolic knots in $L(49,18)$ with orientation reversing homeomorphic complements. They also give very good reasons for the conjecture that all knots in hyperbolic $3$-manifolds are determined by their oriented complements. Recently Ichihara, Jong and Masai~\cite{IJ} found examples of knots in hyperbolic manifolds with orientation-reversing homeomorphic complements.

Altogether this leads to the still open oriented knot complement conjecture in general manifolds:

\begin{con}[Oriented knot complement conjecture]
If $K_1$ and $K_2$ are knots in a closed oriented $3$-manifold $M$ with orientation-pre\-ser\-ving homeomorphic complements (not homeomorphic to $S^1\times D^2$), then the knots are orientation-preserving equivalent.
\end{con}

With Theorem~\ref{thm:cosmetic} and Remark~\ref{rem:orientedcosmetic} this is equivalent to the cosmetic surgery conjecture formulated in~\cite{BHW}:

\begin{con}[Oriented cosmetic surgery conjecture]
Exotic cosmetic surgeries (not resulting in a Lens space) are never orientation-preserving (or \textbf{truly}) cosmetic.
\end{con}

Many of the mentioned results rely on the classification of all Dehn surgeries in a solid torus resulting again in a solid torus by Berge~\cite{B} and Gabai~\cite{Ga2}. But in the last years also Heegaard--Floer homology turned out to be a very useful tool to study such questions. For Example, Wang~\cite{Wan} showed that there are no exotic Dehn surgeries along Seifert genus-$1$ knots in $S^3$. The same holds for non-trivial algebraic knots in $S^3$ by~\cite{Ra1}. Finally, in~\cite{Gai} and~\cite{Ra2} it is shown that knots in $L$-space homology spheres are determined by their complements. In particular, knots in the Poincar\'{e} sphere are determined up to orientation-preserving equivalence by the oriented homeomorphism types of their complements.


	\section{The Legendrian knot complement problem in general manifolds}	
	\label{section:other}

As far as we know, nothing is known about this in the contact setting. The question is which results from the topological setting generalize to the contact setting and where are the differences. In the foregoing section, we presented examples of non-equivalent knots in Lens spaces with homeomorphic exteriors. The first interesting question is if one can generalize these examples to the contact setting. For that we start with two standard neighborhoods of Legendrian knots and glue them together along their boundaries to obtain a contact lens space. Then the exteriors of these Legendrian knots in the lens space are contactomorphic and we will check if the Legendrian knots are equivalent.

\begin{ex}  [Gluing standard neighborhoods of Legendrian knots]
Consider a Legendrian knot $K$ with $\operatorname{tb}(K)=n$. A standard neighborhood of $K$ is given by
\begin{equation*}
\big(S^1\times D^2,\ker(\cos n\theta \,dx-\sin n\theta\,dy)\big).
\end{equation*}
The surface longitude is given by $\lambda=S^1\times \{p\}$ and the contact longitude by $\lambda+n\mu$. Take two copies $(V_1,\xi_1)$ and $(V_2,\xi_2)$ of this standard neighborhood and glue them together along their boundaries to obtain the lens space $L(p,-q)$ as follows:
\\
\hfill\\
\begin{xy}
(20,10)*+{L(p,-q)}="a";(30,10)*+{=}="a1"; (40,10)*+{V_1}="c"; (60,10)*+{+}="d"; (80,10)*+{V_2}="e";(95,10)*+{\big/_\sim}="e1";%
(40,5)*+{\mu_1}="f"; (80,5)*+{q\mu_2+p\lambda_2,}="g";%
(40,0)*+{\lambda_1}="h"; (80,0)*+{r\mu_2+s\lambda_2,}="i";%
{\ar@{|->} "f";"g"};
{\ar@{|->} "h";"i"};
\end{xy}\\
\hfill\\
where $qs-pr=-1$. It is a standard fact that the contact structures on the solid tori fit together to a contact structure in the new lens space $L(p,-q)$ if the gluing map sends the contact longitude of $V_1$ to the contact longitude of $V_2$. (The contact structure will be coorientable because the gluing map is chosen to be orientation reversing.) This leads to the conditions $r=n_2-n_1q$ and $s=1-n_1p$. Putting this together it follows that $q=n_2p-1$. 

In particular, it follows that if a contact lens space is obtained by gluing together two standard neighborhoods of Legendrian knots, then the lens space is of the form $L(p,1-n_2p)=L(p,1)$. We have constructed two Legendrian knots $K_1$ and $K_2$ in a contact $L(p,1)$ with contactomorphic exteriors. However, we have seen in Example~\ref{ex:countergeneral} that in this case the knots $K_1$ and $K_2$ are topologically equivalent by a diffeomorphism interchanging the two Heegaard tori $V_1$ and $V_2$ and in fact this diffeomorphism preserves also the contact structure on $L(p,1)$. It follows that the Legendrian knots $K_1$ and $K_2$ in the contact $L(p,1)$ are equivalent.
\end{ex}

In conclusion, we do \textbf{not} get obvious counterexamples to the Legendrian knot complement problem in general contact manifolds. (Recall, that in all other known topological examples of non-equivalent knots with homeomorphic exteriors the orientation on the exteriors is reversed and thus cannot work in contact geometry.)

\begin{prob}[The Legendrian knot exterior problem in general manifolds]
Is a Legendrian knot $K$ in a general contact $3$-manifold $(M,\xi)$ determined by the contactomorphism type of its exterior $(M\setminus\mathring{\nu K},\xi)$?
\end{prob}

A first step to study the Legendrian knot exterior problem in general contact manifolds would be a generalization of Theorem~\ref{thm:cosmetic}. The generalization of the implication $(2)\Rightarrow(1)$ one can prove similarly to the proof of Theorem~\ref{thm:contact1}.

\begin{lem}[Criterion for the Legendrian knot exterior problem]
Let $K$ be a Legendrian knot in a contact manifold $(M,\xi)$, such that there is no non-trivial contact Dehn surgery along $K$ resulting again in $(M,\xi)$. Then the equivalence type of $K$ is determined by the contactomorphism type of its exterior.
\end{lem}

But the other implication does not generalize. The problem is again that the new glued-in solid torus is in general not a standard neighborhood of a Legendrian knot as explained in Example~\ref{ex:counterlink}.

So it is not directly clear if exotic cosmetic contact surgeries leads to counterexamples to the Legendrian knot exterior problem. But their existence is interesting in its own. The following example was obtained earlier by Geiges and Onaran (see also~\cite{GO}).

\begin{ex}  [Exotic contact surgeries]
Consider the two different contact Dehn surgeries along the Legendrian unknot $U$ with $\operatorname{tb}=-1$ in $(S^3,\xi_{st})$ with contact surgery coefficients $r_1=-3/2$ and $r_2=-2/3$ (see Figure~\ref{fig:lensspace}). By expressing the surgery coefficients with respect to the surface longitude $\lambda_s$ given by the meridian of the complementary solid torus one sees that the resulting manifolds are topologically the manifolds from Example~\ref{ex:countergeneral}, so the homeomorphic lens spaces $L(5,2)$ and $L(5,3)$.
\begin{figure}[htbp] 
\centering
\def\svgwidth{0.99\columnwidth}
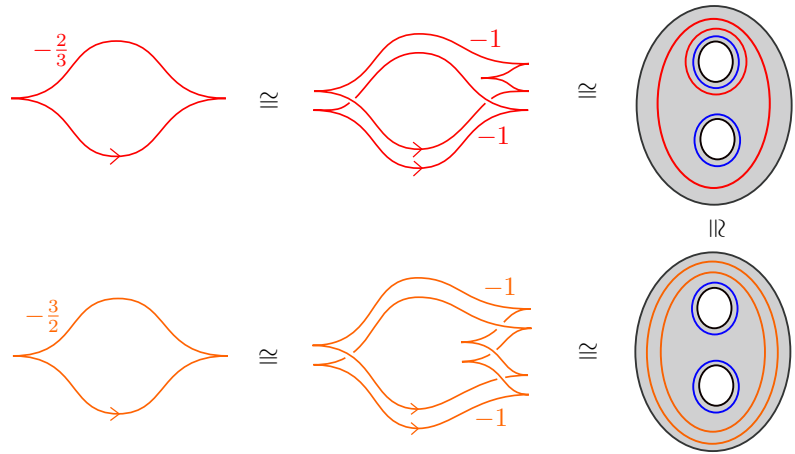
\caption{Two surgery diagrams of the same tight contact structure $\xi$ on $L(5,2)$}
\label{fig:lensspace}
\end{figure} 

The next thing we want to show is that the two contact surgeries are unique and represent the same contact manifold. Therefore one uses the same approach as in Example~\ref{ex:unique}. The same argument as in Example~\ref{ex:unique} shows that the two surgery diagrams both represent a unique contact structure $\xi$ on $L(5,2)$. The from the algorithm~\cite[Section 1]{DGS} resulting $(\pm1)$-contact surgery diagrams are shown in Figure~\ref{fig:lensspace} in the middle (for both contact surgeries is only one of the two possible stabilizations drawn). Both surgery diagrams contain only $(-1)$-contact surgeries, so the resulting contact structures have to be tight. But the classification of tight contact structures on lens spaces~\cite[Theorem~2.1]{H} says that on $L(5,2)$ there are two non-contactomorphic tight contact structures, so in this case, this is not enough to conclude that these surgery diagrams represent the same contact manifold.

To show that the contact structures of the two different surgeries are really contactomorphic one changes the contact Dehn surgery diagrams into compatible open book decompositions as in the proof of Lemma~\ref{lem:rolfsentwist} (for details see for example~\cite{O}). The resulting open books are shown in Figure~\ref{fig:lensspace} on the right. All colored curves represent right-handed Dehn twists, the blue ones correspond to the monodromy of the open book decomposition of $(S^3,\xi_{st})$ and the red and orange curves represent the Dehn twists corresponding to the same colored Legendrian links. All Dehn twists together represent the monodromy of $(L(5,2),\xi)$. By interchanging the holes, these two open books are the same, and therefore represent the same contact manifolds. 

These cosmetic contact surgeries are exotic because the corresponding topological cosmetic surgeries are exotic (i.e. there is no homeomorphism of the exterior of the unknot that maps one slope to the other).

With exactly the same methods one can get many other examples of this kind. For example, by contact Dehn surgery along $U$ with contact surgery coefficients $r_1=-2/5$ and $r_2=-3/4$ one gets contactomorphic contact structures on the lens spaces $L(7,5)$ and $L(7,3)$.
\end{ex}

An easy consequence of Theorem~\ref{thm:contact2} is that the only possible candidates for Legendrian knots in $(S^3,\xi_{st})$ admitting a (non-trivial) contact surgery resulting again in $(S^3,\xi_{st})$ are the Legendrian unknots of type $U_n$.

That cosmetic contact surgeries along such knots really exist was shown in Example~\ref{ex:surgeryUnknot}. It is also possible to show that every Legendrian unknot of type $U_n$ admits (infinitely many) cosmetic contact surgeries resulting again in $(S^3,\xi_{st})$. By computing explicitly the $d_3$-invariants (see~\cite{DGS,DK}) of all other contact surgeries one can classify all rational contact Dehn surgeries along a single Legendrian unknot resulting in a $S^3$ with an arbitrary contact structure~\cite[Section~5.4]{Ke17a}.

\begin{prob}[Cosmetic contact surgeries]
Is it possible to classify cosmetic contact surgeries in other contact manifolds? Can one find examples of exotic cosmetic contact surgeries not resulting in Lens spaces?
\end{prob}

Moreover, in Section \ref{counterexample} we showed that all topological examples of links not determined by their complements, where one does a composition of $(-1)$-Rolfsen twist along unknot components, works also for the contact case. But in the topological category there are also examples that do not arise in this way. Teragaito~\cite{Te} and Ichihara~\cite{I} construct, building up on unpublished work of Berge, links in $S^3$ with no unknot components, such that non-trivial Dehn surgery along that link leads again to $S^3$. Similar to the arguments in Theorem~\ref{thm:cosmetic} this leads to links not determined by their complements. Gordon proves in~\cite{Gor} that for every such link (not determined by its complement but without unknot components) there exist only finitely many other links with homeomorphic complements. Finally in~\cite{MOS} the reverse question is considered. They study links that \textbf{are} determined by their complements.

\begin{prob}[The Legendrian link exterior problem]
Which of these statements hold also for Legendrian links in contact manifolds?
\end{prob}



\end{document}